\documentclass[12pt]{amsart}

\usepackage{amssymb}

\usepackage{enumitem}

\usepackage{graphicx}
\usepackage[all]{xy}

\makeatletter
\@namedef{subjclassname@2020}{%
  \textup{2020} Mathematics Subject Classification}
\makeatother

\usepackage[T1]{fontenc}

\newtheorem{theorem}{Theorem}[section]
\newtheorem{corollary}[theorem]{Corollary}
\newtheorem{lemma}[theorem]{Lemma}
\newtheorem{proposition}[theorem]{Proposition}

\theoremstyle{definition}
\newtheorem{definition}[theorem]{Definition}
\newtheorem{remark}[theorem]{Remark}
\newtheorem{example}[theorem]{Example}

\numberwithin{equation}{section}

\newcommand{\modu}{\operatorname{mod}}
\newcommand{\Hom}{\operatorname{Hom}}
\newcommand{\Gen}{\operatorname{Gen}}
\newcommand{\Cogen}{\operatorname{Cogen}}

\newcommand{\rad}{\operatorname{rad}}
\newcommand{\Ext}{\operatorname{Ext}}
\newcommand{\inj}{\operatorname{inj}}
\newcommand{\add}{\operatorname{add}}

\newcommand{\proj}{\operatorname{proj}}
\newcommand{\Ker}{\operatorname{Ker}}
\newcommand{\Coker}{\operatorname{Coker}}
\newcommand{\Ima}{\operatorname{Im}}

\begin{document}


\baselineskip=17pt


\title[Support $\tau$-tilting modules and recollements]{Support $\tau$-tilting modules and recollements}

\author[X. Ma]{Xin Ma}
\address[X. Ma]{College of Science\\
 Henan University of Engineering\\
  451191 Zhengzhou, P.R. China}

\author[Z. Xie]{Zongzhen Xie}
\address[Z. Xie]{Department of Mathematics and Computer Science\\
 School of Biomedical Engineering and Informatics\\
  Nanjing Medical University\\
  211166 Nanjing, P.R. China}

\author[T. Zhao]{Tiwei Zhao$^\ast$}
\address[T. Zhao]{School of Mathematical Sciences\\
 Qufu Normal University\\
  273165 Qufu, P.R. China}
\email{tiweizhao@qfnu.edu.cn}

\date{}
\thanks{$^\ast$Corresponding author}

\begin{abstract}
Let $(\modu \Lambda',\modu  \Lambda,\modu  \Lambda'')$ be a recollement of abelian  categories for artin algebras $\Lambda'$, $\Lambda$ and $\Lambda''$. Under certain conditions, we present an explicit construction of gluing of (support) $\tau$-tilting modules in $\modu  \Lambda$ with respect to (support) $\tau$-tilting modules in $\modu  \Lambda'$ and $\modu \Lambda''$ respectively. On the other hand, we study the construction of (support) $\tau$-tilting modules in $\modu \Lambda'$ and $\modu  \Lambda''$ obtained from (support) $\tau$-tilting modules in $\modu  \Lambda$.
\end{abstract}



\maketitle

\section{Introduction} 

Recollements of abelian and triangulated categories were introduced by Be{\u\i}linson, Bernstein and Deligne \cite{BBD81F} in
connection with derived categories of sheaves on topological spaces with the idea that one triangulated category may be glued together from two others, which play an important role in
representation theory of algebras \cite{BBD81F,BARI07H,CXW12A,HY14R,JJP65S,QYYHY16R,ZP13G,Zhe}. Recollements of abelian categories and triangulated categories
are closely related, and they possess similar properties in many aspects.

Recently, gluing techniques with respect to a recollement of triangulated or abelian categories have been investigated; for instance, for a given recollement of triangulated categories, Chen \cite{CJM13C} glued cotorsion pairs in a recollement; Liu, Vit\'oria and Yang presented the construction of gluing of silting objects \cite{LQHVJ14G} in a recollement, which follows from the bijection between the equivalence class of silting objects and that of bounded co\mbox{-}t\mbox{-}structures whose co-hearts are additively generated by one object (see \cite{MHOSVEC13A}). For a recollement of abelian categories, Parra and Vit\'oria \cite{PCVJ17P} glued some basic properties of abelian categories (well-poweredness, Grothendieck's axioms, existence of a generator) in a recollement. For more references, see \cite{MXHZY, MZ, Zh}.

Tilting theory is one of  important tools in the representation theory of artin algebras. The classical tilting module was introduced by Brenner and Butler \cite{BB79G}, and Happel and Ringel \cite{HDRCM82T}.  One of the important properties
of tilting modules is that when an almost complete tilting module has two complements, and thus one can make a mutation process for it. However, an almost complete tilting module has not always two complements. Thus, to make mutation always possible, it is desirable to enlarge
the class of tilting modules in order to get the more
regular property. In order to accomplish this goal,  Adachi, Iyama and Reiten introduced $\tau$-tilting theory in  \cite{ATIO14t}, and recently, Yang and Zhu in \cite{YZ} generalized the results of Adachi-Iyama-Reiten further in the setting of triangulated categories. In their study, support $\tau$-tilting modules are very closely related with torsion classes, torsionfree classes, and cluster tilting subcategories; for instance, there are the bijection between the class of basic support $\tau$-tilting modules and that of functorially finite torsion classes (see \cite{ATIO14t,GH,LX,YZ,ZZ}).

Applications of gluing techniques in recollements lead us to ask whether a (support) $\tau$-tilting module can be ``glued together'' from two other (support) $\tau$-tilting modules in a recollement of module categories; or equivalently, whether a functorially finite torsion class can be ``glued together'' from two other functorially finite torsion classes in a recollement of abelian categories. In this paper, we answer these questions, which provide some methods and strategies for constructing (support) $\tau$-tilting modules.

In this paper, we use $\modu \Lambda$ to denote the category of finitely generated left $\Lambda$-module for an artin algebra $\Lambda$. Let $(\mathcal{A},\mathcal{B},\mathcal{C})$
be a recollement of abelian categories. In particular, let $\Lambda'$, $\Lambda$ and $\Lambda''$ be artin algebras such that
$$\xymatrix{\modu \Lambda'\ar[rr]!R|-{i_{*}}&&\ar@<-2ex>[ll]!R|-{i^{*}}
\ar@<2ex>[ll]!R|-{i^{!}}\modu \Lambda
\ar[rr]!L|-{j^{*}}&&\ar@<-2ex>[ll]!L|-{j_{!}}\ar@<2ex>[ll]!L|-{j_{*}}
\modu \Lambda''}$$
is a recollement of abelian categories.
Our main results are the following
\begin{theorem}\label{thm1}
 Let $T'$ and $T''$ be support $\tau$-tilting modules in $\modu \Lambda'$ and $\modu \Lambda''$ respectively, and $(\mathcal{T},\mathcal{F})$ a glued torsion pair in $\modu \Lambda$ with respect to $(\Gen T', T'^{\perp_{0}})$ and $(\Gen T'', T''^{\perp_{0}})$.
  \begin{itemize}
\item[(1)] If $i^{!}$ is exact and $i_{*}i^{!}(\mathcal{T})\subseteq\mathcal{T}$, then there is a support $\tau$-tilting $\Lambda$-module $T$ such that $(\mathcal{T},\mathcal{F})=(\Gen T, T^{\perp_{0}})$.
    \item[(2)] If $i^{*}$ is exact and $i_{*}i^{*}(\mathcal{F})\subseteq\mathcal{F}$, then there is a support $\tau$-tilting $\Lambda$-module $T$ such that $(\mathcal{T},\mathcal{F})=(\Gen T, T^{\perp_{0}})$.
  \end{itemize}
\end{theorem}

On the other hand,

\begin{theorem}
Let $T$ be a support $\tau$-tilting $\Lambda$-module and $(\mathcal{T},\mathcal{F}):=(\Gen T,T^{\perp_{0}})$ a torsion pair induced by $T$ in $\modu \Lambda$.
\begin{itemize}
\item[(1)] If $i^{*}$ has a left adjoint  or $i^{!}$ has a right adjoint, then there is a support $\tau$-tilting module $T'$ in $\modu \Lambda'$ such that $(\Gen T',T'^{\perp_{0}})=(i^{*}(\mathcal{T}),i^{!}(\mathcal{F}))$.
\item[(2)]
\begin{itemize}
\item[(a)] If $j_{*}j^{*}(\mathcal{F})\subseteq\mathcal{F}$, then there is a support $\tau$-tilting module $T''$ in $\modu \Lambda''$ such that $(\Gen T'',T''^{\perp_{0}})=(j^{*}(\mathcal{T}),j^{*}(\mathcal{F}))$.
\item[(b)] If $j_{*}$ is exact and $j_{*}j^{*}(\mathcal{T})\subseteq\mathcal{T}$, then there is a support $\tau$-tilting module $T''$ in $\modu \Lambda''$ with respect to $j^{*}(\mathcal{T})$.
     Moreover, $j_{*}j^{*}(\mathcal{F})\subseteq\mathcal{F}$ if and only if $\Gen T''=j^{*}(\mathcal{T})$ and $T''^{\perp_{0}}=j^{*}(\mathcal{F})$.
\item[(c)] If $j_{!}$ is exact and $j_{!}j^{*}(\mathcal{F})\subseteq\mathcal{F}$, then there is a support $\tau$-tilting module $T''$ in $\modu \Lambda''$ with respect to $j^{*}(\mathcal{F})$. Moreover, $j_{*}j^{*}(\mathcal{F})\subseteq\mathcal{F}$ if and only if $\Gen T''=j^{*}(\mathcal{T})$ and $T''^{\perp_{0}}=j^{*}(\mathcal{F})$.
\end{itemize}
\end{itemize}
\end{theorem}

The above results describe methods to construct support $\tau$-tilting modules in a recollement.
As a special case for support $\tau$-tilting modules, we can construct $\tau$-tilting modules in a recollement by this way.
This paper is organized as follows.

In Section \ref{pre}, we give some terminologies and some preliminary results.

In Section \ref{tor}, we first study the construction of functorially finite torsion classes (resp. functorially finite torsionfree classes) in a recollement of abelian categories.  Under certain conditions, we give an explicit construction of  functorially finite torsion classes (resp. functorially finite torsionfree classes) in $\mathcal{B}$ from functorially finite torsion classes (resp. functorially finite torsionfree classes) in $\mathcal{A}$ and $\mathcal{C}$ respectively.
Using a relation between functorially finite torsion classes and support $\tau$-tilting modules,  we show the existence of support $\tau$-tilting modules in $\modu \Lambda$ from support $\tau$-tilting modules in $\modu \Lambda'$ and $\modu \Lambda''$ respectively.
On the other hand, under certain conditions, we can use functorially finite torsion classes (resp. functorially finite torsionfree classes) in $\mathcal{B}$ to induce functorially finite torsion classes (resp. functorially finite torsionfree classes) in $\mathcal{A}$ and $\mathcal{C}$. By the same reason, we show the existence of support $\tau$-tilting modules in $\modu \Lambda'$ and $\modu \Lambda''$ respectively from a support $\tau$-tilting module in $\modu \Lambda$.

In Section 4, we focus on $\tau$-tilting modules, which are special cases of support $\tau$-tilting modules. Under certain conditions, we give a construction of sincere functorially finite torsion classes in $\modu \Lambda$ from sincere functorially finite torsion classes in $\modu \Lambda'$ and $\modu \Lambda''$ respectively;
on the other hand, we show that sincere functorially finite torsion classes in $\mathcal{B}$ can induce sincere functorially finite torsion classes in $\mathcal{A}$ and $\mathcal{C}$.
Thus, as similar arguments to Section 3, we can obtain the corresponding results for $\tau$-tilting modules.

Finally, in Section 5, we give some examples to illustrate our obtained results.

Throughout this paper, unless otherwise stated, $\mathcal{A}$, $\mathcal{B}$ and $\mathcal{C}$ are abelian categories, and all subcategories are full, additive and closed under isomorphisms. All algebras are artin algebras and all modules are finitely generated left modules.

\section{Preliminaries}\label{pre}
We recall the notion of recollements of abelian categories.
\begin{definition}\label{def-2.1}
{(\cite{VFTP04C}) A \emph{recollement}, denoted by ($\mathcal{A},\mathcal{B},\mathcal{C}$), of abelian categories is a diagram
$$\xymatrix{\mathcal{A}\ar[rr]!R|{i_{*}}&&\ar@<-2ex>[ll]!R|{i^{*}}\ar@<2ex>[ll]!R|{i^{!}}\mathcal{B}
\ar[rr]!L|{j^{*}}&&\ar@<-2ex>[ll]!L|{j_{!}}\ar@<2ex>[ll]!L|{j_{*}}\mathcal{C}}$$
of abelian categories and additive functors such that
\begin{enumerate}
\item[(1)] ($i^{*},i_{*}$), ($i_{*},i^{!}$), ($j_{!},j^{*}$) and ($j^{*},j_{*}$) are adjoint pairs.
\item[(2)] $i_{*}$, $j_{!}$ and $j_{*}$ are fully faithful.
\item[(3)] $\Ima i_{*}=\Ker j^{*}$.
\end{enumerate}}
\end{definition}

We list some properties of recollements (see \cite{ VFTP04C,MXHZY,PC14H,PCSO14G,PCVJ14R}), which will be used in the sequel.

\begin{lemma}\label{lem-rec}
Let ($\mathcal{A},\mathcal{B},\mathcal{C}$) be a recollement of abelian categories.
\begin{enumerate}
\item[(1)] $i^{*}j_{!}=0= i^{!} j_{*}$.
\item[(2)] The functors $i_{*}$, $j^{*}$ are exact, and the functors $i^{*}$, $j_{!}$ are right exact, and the functors $i^{!}$, $j_{*}$ are left exact.
\item[(3)] The natural transformations $\xymatrix@C=15pt{i^{*}i_{*}\ar[r]&1_{\mathcal{A}}}$, $\xymatrix@C=15pt{1_{\mathcal{A}}\ar[r]&i^{!}i_{*}}$, $\xymatrix@C=15pt{1_{\mathcal{C}}\ar[r]&j^{*}j_{!}}$, and $\xymatrix@C=15pt{j^{*}j_{*}\ar[r]&1_{\mathcal{C}}}$ are natural isomorphisms. Moreover, the functors $i^{*}$, $i^{!}$ and $j^{*}$ are dense.
    \item[(4)] Let $B\in \mathcal{B}$. There exist exact sequences
$$\xymatrix@C=15pt{0\ar[r]&i_{*}(A)\ar[r]&j_{!}j^{*}(B)\ar[r]^-{\epsilon_{B}}&B\ar[r]&i_{*}i^{*}(B)\ar[r]&0,}$$
$$\xymatrix@C=15pt{0\ar[r]&i_{*}i^{!}(B)\ar[r]^-{\lambda_{B}}&B\ar[r]^-{\eta_{B}}&j_{*}j^{*}(B)\ar[r]&i_{*}(A')\ar[r]&0}$$
in $\mathcal{B}$ with $A,A'\in \mathcal{A}$.
\item[(5)] Let $B\in\mathcal{B}$. If $i^{*}$ is exact, then we have the following exact sequence
            $$\xymatrix@C=15pt{0\ar[r]&j_{!}j^{*}(B)\ar[r]^-{\epsilon_{B}}&
            B\ar[r]&i_{*}i^{*}(B)\ar[r]&0}.$$

If $i^{!}$ is exact, then we have the following exact sequence
            $$\xymatrix@C=15pt{0\ar[r]&i_{*}i^{!}(B)\ar[r]^-{\lambda_{B}}&B\ar[r]^-{\eta_{B}}&
            j_{*}j^{*}(B)\ar[r]&0}.$$
\item[(6)] If $i^{*}$ is exact, then $i^{!}j_{!}=0$; and if $i^{!}$ is exact, then $i^{*}j_{*}=0$.
\end{enumerate}
\end{lemma}

Recall that a subcategory $\mathcal{D}$ of an abelian category $\mathcal{A}$ is called a \emph{Serre subcategory} if for any exact sequence $\xymatrix@C=15pt{0\ar[r]&A\ar[r]&B\ar[r]&C\ar[r]&0}$ in $\mathcal{A}$, $B$ is in $\mathcal{D}$ if and only if both $A$ and $C$ are in $\mathcal{D}$.

\begin{lemma}\label{lem-serre}
Let $(\mathcal{A},\mathcal{B},\mathcal{C})$ be a recollement of abelian categories.
\begin{itemize}
\item[(1)] $i_{*}(\mathcal{A})$ is a Serre subcategory of $\mathcal{B}$.
\item[(2)] If $i^{*}$ is exact, then $\Ker i^{*}=\Ima j_{!}$. Moreover,  $j_{!}(\mathcal{C})$ is a Serre subcategory of $\mathcal{B}$.
\item[(3)] If $i^{!}$ is exact, then $\Ker i^{!}=\Ima j_{*}$. Moreover,  $j_{*}(\mathcal{C})$ is a Serre subcategory of $\mathcal{B}$.
\end{itemize}
\end{lemma}
\begin{proof}
(1) Apply \cite[Proposition 2.8]{PCVJ14R}.

(2) The inclusion $\Ima j_{!}\subseteq\Ker i^{*}$ is a consequence of  Lemma \ref{lem-rec}(1).

Conversely, let $X$ be any object in $\Ker i^{*}$, that is, $i^{*}(X)=0$. Since $i^{*}$ is exact, by Lemma \ref{lem-rec}(5), there is an exact sequence
  $$\xymatrix@C=15pt{0\ar[r]&j_{!}j^{*}(X)\ar[r]^-{\epsilon_{X}}&
            X\ar[r]&i_{*}i^{*}(X)\ar[r]&0}$$
            in $\mathcal{B}$.
            Thus $X\cong  j_{!}j^{*}(X)\in\Ima j_{!}$.

            The second assertion is trivial.

(3) It is similar to (2).
  \end{proof}

\begin{lemma}\label{lem-2.6}
Let $(\mathcal{A},\mathcal{B},\mathcal{C})$ be a recollement of abelian categories.
\begin{itemize}
\item[(1)] If $i^{*}$ is exact, then $j_{!}$ is exact.
\item[(2)] If $i^{!}$ is exact, then $j_{*}$ is exact.
\end{itemize}
\end{lemma}
\begin{proof}
(1) Let
$$\xymatrix@C=15pt{0\ar[r]&X\ar[r]&Y\ar[r]&Z\ar[r]&0}$$
 be an exact sequence in $\mathcal{C}$. Since $j_{!}$ is right exact by Lemma \ref{lem-rec}(2), applying $j_{!}$ to the above exact sequence yields an exact sequence
\begin{align}\label{2.1}
\xymatrix@C=15pt{0\ar[r]&C\ar[r]&j_{!}(X)\ar[r]&j_{!}(Y)\ar[r]&j_{!}(Z)\ar[r]&0}
\end{align}
in $\mathcal{B}$.
  Since $j^{*}$ is exact and $j^{*}j_{!}\cong 1_{\mathcal{   C}}$ by Lemma \ref{lem-rec}(2) and (3), applying $j^{*}$ to the above exact sequence yields $j^{*}(C)=0$. Since $\Ima i_{*}=\Ker j^{*}$, there exists an object $C'$ in $\mathcal{A}$ such that $C\cong i_{*}(C')$.
Since $i^{*}$ is exact (by assumption) and $i^{*}j_{!}=0$ (by Lemma \ref{lem-rec}(1)), applying $i^{*}$ to the exact sequence (\ref{2.1}) yields $i^{*}(C)=0$. It follows that $C'\cong i^{*}i_{*}(C')\cong i^{*}(C)=0$ and $C=0$. Thus $j_{!}$ is exact.

(2) It is similar to (1).
\end{proof}

Let $\mathcal{A}$ be an abelian category with enough projective objects and injective objects. We denote by $\proj \mathcal{A}$ (resp. $\inj \mathcal{A}$) the subcategory of $\mathcal{A}$ consisting of all projective (resp. injective) objects in $\mathcal{A}$.

Let $\mathcal{D}$ be a class of objects in $\mathcal{A}$. We denote by $\add \mathcal{D}$ the subcategory of $\mathcal{A}$ consisting of direct summands of finite direct sums of objects in $\mathcal{D}$. In particular, if $\mathcal{D}$ has only a single object $D$, we denote it by $\add D$.

We need the following easy and useful observations.

\begin{proposition}\label{C-enough}
Let $(\mathcal{A},\mathcal{B},\mathcal{C})$ be a recollement of abelian categories.
\begin{itemize}
\item[(1)] If $\mathcal{B}$ has enough projective objects, then the functor $\xymatrix@C=15pt{i^{*}: \mathcal{B}\ar[r]&\mathcal{A}}$ preserves projective objects. In this case,  $\mathcal{A}$ has enough projective objects and $\proj \mathcal{A}=\add (i^{*}(\proj \mathcal{B}))$. Dually, if $\mathcal{B}$ has enough injective objects, then the functor $\xymatrix@C=15pt{i^{!}: \mathcal{B}\ar[r]&\mathcal{A}}$ preserves injective objects. In this case,  $\mathcal{A}$ has enough injective objects and $\inj \mathcal{A}=\add (i^{!}(\inj \mathcal{B}))$.
\item[(2)] If $\mathcal{C}$ has enough projective objects, then the functor $\xymatrix@C=15pt{j_{!}: \mathcal{C}\ar[r]&\mathcal{B}}$ preserves projective objects. Dually, if $\mathcal{C}$ has enough injective objects, then the functor $\xymatrix@C=15pt{j_{*}: \mathcal{C}\ar[r]&\mathcal{B}}$ preserves injective objects.
\item[(3)] If $j_{*}$ is exact and $\mathcal{B}$ has enough projective objects, then $j^{*}$ preserves projective objects. In this case,  $\mathcal{C}$ has enough projective objects and $\proj \mathcal{C}=\add (j^{*}(\proj \mathcal{B}))$.
\item[(4)] If $j_{!}$ is exact and $\mathcal{B}$ has enough injective objects, then $j^{*}$ preserves injective objects. In this case,  $\mathcal{C}$ has enough injective objects and $\inj \mathcal{C}=\add (j^{*}(\inj \mathcal{B}))$.
\end{itemize}
\end{proposition}

\begin{proof}
The assertions (1) and (2) follow from \cite[Remark 2.5]{PC14H}. The assertions (3) and (4) are similar to (1) and (2).
\end{proof}

\begin{proposition}\label{prop-exact}
Let $(\mathcal{A},\mathcal{B},\mathcal{C})$ be a recollement of abelian categories.
\begin{itemize}
\item[(1)] Assume that $\mathcal{A}$ and $\mathcal{B}$ have enough injective objects. If $i^{*}$ is exact, then $i_{*}$ and $j^{*}$ preserve injective objects.
\item[(2)] Assume that $\mathcal{A}$ and $\mathcal{B}$ have enough projective objects. If $i^{!}$ is exact, then $i_{*}$ and $j^{*}$ preserve projective objects.
\end{itemize}
\end{proposition}
\begin{proof}
(1) The assertion that $i_{*}$ preserves injective objects is similar to Proposition \ref{C-enough}. Since $i^{*}$ is exact by assumption, $j_{!}$ is exact by Lemma \ref{lem-2.6}(1). Then by Proposition \ref{C-enough}(4),  $j^{*}$ preserves injective objects.

(2) It is similar to (1).
\end{proof}

\begin{lemma}{\rm(\cite[Lemma 3.10]{LM17G})}\label{lem-adjoint}
Let $\mathcal{A}$ and $\mathcal{B}$ be abelian categories with enough projective objects. If $\xymatrix@C=15pt{F: \mathcal{A}\ar[r]&\mathcal{B}}$ is an exact functor admitting a right adjoint $G$, and if $F$ preserves projective objects, then  $$\Ext^{i}_{\mathcal{B}}(F(X),Y)\cong \Ext_{\mathcal{A}}^{i}(X,G(Y))$$
for any $X\in \mathcal{A}$ and $Y\in\mathcal{B}$, and any $i\geq 1$.
\end{lemma}

\begin{proposition}\label{prop-adjoint}
Let $(\mathcal{A},\mathcal{B},\mathcal{C})$ be a recollement of abelian categories and $n$ any positive integer.
\begin{itemize}
\item[(1)] If $\mathcal{B}$ has enough projective objects and $i^{*}$ is exact, then $\Ext^{n}_{\mathcal{A}}(i^{*}(X),Y)\cong \Ext_{\mathcal{B}}^{n}(X,i_{*}(Y))$.
\item[(2)] If $\mathcal{A}$ has enough projective objects and $i^{!}$ is exact, then $\Ext^{n}_{\mathcal{B}}(i_{*}(X),Y)\cong \Ext_{\mathcal{A}}^{n}(X,i^{!}(Y))$.
\item[(3)] If $\mathcal{C}$ has enough projective objects and $j_{!}$ is exact, then $\Ext^{n}_{\mathcal{B}}(j_{!}(X),Y)\cong \Ext_{\mathcal{C}}^{n}(X,j^{*}(Y))$.
\item[(4)] If $\mathcal{B}$ has enough projective objects and $j_{*}$ is exact, then $\Ext^{n}_{\mathcal{C}}(j^{*}(X),Y)\cong \Ext_{\mathcal{B}}^{n}(X,j_{*}(Y))$.
\end{itemize}
\end{proposition}
\begin{proof}
Apply Propositions \ref{C-enough} and \ref{prop-exact}, and Lemma \ref{lem-adjoint}.
\end{proof}

Now we recall the following notions.
\begin{definition}
{(\cite{DSE66A}) A pair of subcategories $(\mathcal{T},\mathcal{F})$ of an abelian category $\mathcal{A}$ is called a \emph{torsion  pair}
if the following conditions are satisfied.
\begin{itemize}
\item[(1)] $\Hom_{\mathcal{A}}(\mathcal{T},\mathcal{F})=0$; that is, $\Hom_{\mathcal{A}}(X,Y)=0$
for any $X\in\mathcal{T}$ and $Y\in\mathcal{F}$.
\item[(2)] For any object $A\in \mathcal{A}$, there exists an exact sequence
\begin{align*}
\xymatrix@C=15pt{0\ar[r]&X\ar[r]&A\ar[r]&Y\ar[r]&0}
\end{align*}
in $\mathcal{A}$ with $X\in \mathcal{T}$ and $Y\in \mathcal{F}$.
\end{itemize}
In this case, $\mathcal{T}$ is called a \emph{torsion class} in $\mathcal{A}$ and $\mathcal{F}$ is called a \emph{torsionfree class} in $\mathcal{A}$.}
\end{definition}

Let $(\mathcal{T},\mathcal{F})$ be a torsion pair in an abelian category $\mathcal{A}$. Then
\begin{itemize}
\item[(1)] $\mathcal{T}$ is closed under extensions and quotient objects, and is contravariantly finite.
\item[(2)] $\mathcal{F}$ is closed under extensions and subobjects, and is covariantly finite.
\end{itemize}
Moreover,
$$\mathcal{T}={^{\perp_{0}}\mathcal{F}}:=\{M\in\mathcal{A}\mid\Hom_{\mathcal{A}}(M,\mathcal{F})=0\},$$
$$\mathcal{F}={\mathcal{T}^{\perp_{0}}}:=\{M\in\mathcal{A}\mid \Hom_{\mathcal{A}}(\mathcal{T},M)=0\}.$$

Let $\Lambda$ be an artin algebra.
A subcategory $\mathcal{T}$ of $\modu \Lambda$ is a torsion class (resp. torsionfree class) if and only if it is closed under quotient modules (resp. submodules) and extensions. Conversely, any torsion class $\mathcal{T}$ can give rise to a torsion pair $(\mathcal{T},\mathcal{F})$ in $\modu \Lambda$ (see \cite[Proposition VI.1.4]{AISD}).

Now,
let $T$ be a $\Lambda$-module and $\mathcal{D}$ a subcategory of $\modu \Lambda$.

We denote by $\rad T$ the Jacobson radical of $T$ and $|T|$ the number of pairwise nonisomorphic indecomposable direct summands of $T$. In particular, $|\Lambda|$ denotes the number of pairwise nonisomorphic simple modules in $\modu \Lambda$.
We call $T$ an \emph{$\Ext$-projective} object in $\mathcal{D}$ if $\Ext_{\Lambda}^{1}(T,D)=0$ for all $D\in\mathcal{D}$. We use $P(\mathcal{D})$ to denote the direct sum of one copy of each indecomposable $\Ext$-projective object in $\mathcal{D}$.
Taking a submodule chain $\xymatrix@C+15pt{0= T_{0}\subseteq T_{1}\subseteq T_{2}\subseteq \cdots \subseteq T_{n-2}\subseteq T_{n-1}\subseteq T_{n}=T}$ of $T$,
if the modules $T_{i+1}/T_{i}$ are simple for $0\leq i\leq n-1$, then the submodule chain is called a\emph{ composition series} of $T$ and the simple modules $\{T_{i+1}/T_{i}\}_{0\leq i\leq n-1}$ are called \emph{composition factors} of $T$.
Recall that a $\Lambda$-module $T$ is called \emph{sincere} if all simple $\Lambda$-modules appear as a composition factor of $T$.
Moreover,
a subcategory $\mathcal{D}$ is called \emph{sincere} if there exists a $\Lambda$-module $T$ in $\mathcal{D}$ such that $T$ is sincere.
We use $\Gen T$ to denote the class of all modules $M$ in $\modu \Lambda$ generated by $T$, that is,
 \begin{align*}
 \Gen T=\{M\in\modu \Lambda:\text{there exists a nonnegative integer $n$ and} \\ \text{an epimorphism }\xymatrix@C=15pt{T^{n}\ar[r]&M\ar[r]&0} \text{in} \modu \Lambda \}.
 \end{align*}
 The class $\Cogen T$ is defined dually.

We denote by $\tau$ the AR translation (see \cite{AISD} for the definition).

\begin{definition}{(\cite[Definition 0.1]{ATIO14t})
Let $\Lambda$ be an artin algebra and $T$ a $\Lambda$-module.
 \begin{itemize}
 \item[(1)] $T$ is called \emph{$\tau$-rigid} if $\Hom_{\Lambda}(T,\tau T)$=0.
     \item[(2)] $T$ is called \emph{$\tau$-tilting} (resp. \emph{almost complete $\tau$-tilting}) if $T$ is $\tau$-rigid and $|T|=|\Lambda|$ (resp. $|T|=|\Lambda|-1$).
     \item[(3)] $T$ is called \emph{support $\tau$-tilting} if there exists an idempotent $e$ of $\Lambda$ such that $T$ is a $\tau$-tilting $\Lambda/\langle e\rangle$-module.
     \end{itemize}
}
\end{definition}

Let $T$ be a (support) $\tau$-tilting $\Lambda$-module. Following \cite[Theorem 5.8]{AMSSO81A},  $T$ is $\Ext$-projective in $\Gen T$. So by \cite[Lemma VI.1.9]{AISD}, there is a torsion pair $(\Gen T,T^{\perp_{0}})$ in $\modu \Lambda$.

In this case, we say that a (support) $\tau$-tilting module $T$ induces a torsion pair $(\Gen T,T^{\perp_{0}})$ in $\modu \Lambda$.

We fix the following notations:

$\tau$-tilt $\Lambda$: the class of basic $\tau$-tilting modules in $\modu \Lambda$.

$s\tau$-tilt $\Lambda$: the class of basic support $\tau$-tilting modules in $\modu \Lambda$.

$f$-tors $\Lambda$: the class of functorially finite torsion classes in $\modu \Lambda$.

$f$-torsf $\Lambda$: the class of functorially finite torsionfree classes in $\modu \Lambda$.

$sf$-tors $\Lambda$: the class of sincere functorially finite torsion classes in $\modu \Lambda$.

\begin{remark}{(\cite[Theorem 2.7 and Corollary 2.8]{ATIO14t},\cite{SSO84T})}\label{bijection}
Let $\Lambda$ be an artin algebra. Then

\begin{itemize}
 \item[(1)] There are bijections between
 \begin{itemize}
 \item[(a)] the class $f$-\emph{tors} $\Lambda$ of functorially finite torsion classes in $\modu \Lambda$,
    \item[(b)] the class $f$-\emph{torsf} $\Lambda$ of functorially finite torsionfree classes in $\modu \Lambda$,
        \item[(c)] the class $s\tau$-\emph{tilt} $\Lambda$ of basic support $\tau$-tilting modules in $\modu \Lambda$,
            \end{itemize}
where the bijection between (a) and (c)
$$\xymatrix@C=15pt{s\tau\text{-}\emph{tilt}\ \Lambda\ar[r]&\ar[l]f\text{-}\emph{tors}}\Lambda$$
is given by $s\tau$-\emph{tilt} $\Lambda$ $\ni T$ $\longmapsto\Gen T\in f$-\emph{tors} $\Lambda$ and $f$-\emph{tors} $\Lambda$ $\ni \mathcal{T}$ $\longmapsto P(\mathcal{T})\in s\tau$-\emph{tilt} $\Lambda$.
In this case, let $\mathcal{T}$ (resp. $\mathcal{F})$ be a functorially finite torsion (resp. torsionfree) class. We call $T:=P(\mathcal{T})$ a support $\tau$-tilting module with respect to $\mathcal{T}$, moreover $\Gen T=\mathcal{T}$. On the other hand, let $T$ be any support $\tau$-tilting module,
we have  $\add P(\Gen T)=\add T$.
\item[(2)] By taking the same correspondence as above, one has a bijection
$$\xymatrix@C=15pt{\tau\text{-}\emph{tilt}\ \Lambda\ar[r]&\ar[l]sf\text{-}\emph{tors }\Lambda}.$$

\end{itemize}
\end{remark}

The following fact is known. For the reader's convenience, we list it and give the proof.
\begin{lemma}\label{lem-fullf}
Let $\Lambda'$ and $\Lambda$ be artin algebras, and let $\xymatrix@C=15pt{F: \modu \Lambda'\ar[r]&\modu \Lambda}$ be a fully faithful functor. Then
for any object $M$ in $\modu \Lambda'$,  $M$ is indecomposable if and only if $F(M)$ is indecomposable. More generally, $|F(M)|=|M|$.
\end{lemma}
\begin{proof}
Let $M$ be a $\Lambda'$-module. Since $F$ is fully faithful, $\Hom_{\Lambda}(F(M),F(M))\cong \Hom_{\Lambda'}(M,M)$, and thus $M$ is indecomposable if and only if $F(M)$ is indecomposable by \cite[Theorem II.2.2]{AMRISSO95R}. Now assume $M=\oplus_{i=1}^n M_i$ with each $M_i$ indecomposable. Then $F(M)=\oplus_{i=1}^n F(M_i)$. By the first assertion, each $F(M_i)$ is indecomposable.
 Thus $|F(M)|=|M|$. 
\end{proof}

Now, let $\Lambda'$, $\Lambda$ and $\Lambda''$ be artin algebras such that $(\modu \Lambda',\modu \Lambda,\modu \Lambda'')$ is a recollement of abelian categories:
$$\xymatrix{\modu \Lambda'\ar[rr]|-{i_{*}}&&\ar@<-2ex>[ll]|-{i^{*}}\ar@<2ex>[ll]|-{i^{!}}\modu \Lambda
\ar[rr]|-{j^{*}}&&\ar@<-2ex>[ll]|-{j_{!}}\ar@<2ex>[ll]|-{j_{*}}\modu \Lambda''}.$$

The following result gives the form of simple modules in a recollement, which plays an important role in the sequel.
\begin{lemma}\label{lem-simple}
 Let $S'$ and $S''$ be simple modules in $\modu \Lambda'$ and $\modu \Lambda''$ respectively. If $i^{!}$ $($resp. $i^{*})$ is exact, then $i_{*}(S')$ and $j_{*}(S'')$ $($resp. $i_{*}(S')$ and  $j_{!}(S''))$ are simple $\Lambda$-modules.
 In particular, $|\Lambda|=|\Lambda'|+|\Lambda''|$.
\end{lemma}
\begin{proof}
We only prove the case that $i^{!}$ is exact; the other is similar.
Since $i_{*}$ and $j_{*}$ are fully faithful, by Lemma \ref{lem-fullf},  $M$ is indecomposable if and only if $i_{*}(M)$ is indecomposable for any $\Lambda'$-module $M$,
and $j_{*}(N)$ is indecomposable if and only if $N$ is indecomposable for any $\Lambda''$-module $N$.

Because $i_{*}(S')\in\modu \Lambda$, consider the following exact sequence
 $$\xymatrix@C=15pt{0\ar[r]&\rad i_{*}(S')\ar[r]&i_{*}(S')\ar[r]&i_{*}(S')/\rad i_{*}(S')\ar[r]&0}$$
 in $\modu \Lambda$. Since $i_{*}(\modu \Lambda')$ is a Serre subcategory of $\modu \Lambda$ by Lemma \ref{lem-serre}(1), there exist modules $A_{1}, A_{2}\in\modu \Lambda'$ such that $\rad i_{*}(S')\cong i_{*}(A_{1})$ and $i_{*}(S')/\rad i_{*}(S')\cong i_{*}(A_{2})$. Since $\xymatrix@C=15pt{i_{*}: \modu \Lambda'\ar[r]& i_{*}(\modu \Lambda')}$ is an equivalence, there exists an exact sequence
 $$\xymatrix@C=15pt{0\ar[r]&A_{1}\ar[r]&S'\ar[r]&A_{2}\ar[r]&0}$$
 in $\modu \Lambda'$.
Notice that either $A_{1}=0$ or $A_{1}=S'$. If $A_{1}=S'$, then $i_{*}(S')\cong \rad i_{*}(S')$, which is a contradiction. If $A_{1}=0$, then $\rad i_{*}(S')=0$, and so $i_{*}(S')$ is a semisimple $\Lambda$-module.
Moreover, since $i_{*}(S')$ is indecomposable by the fact that $S'$ is indecomposable,  $i_{*}(S')$ is a simple module in $\modu \Lambda$.

Since $i^{!}$ is exact by assumption,  $j_{*}(\modu \Lambda'')$ is a Serre subcategory of $\modu \Lambda$ by Lemma \ref{lem-serre}(3). As a similar argument to the above, $j_{*}(S'')$ is a simple module in $\modu \Lambda$ for any simple $\Lambda''$-module $S''$.
Note that $j_{*}(S'')$ is not isomorphic to $i_{*}(S')$ for simple modules $S''$ and $S'$ in $\modu \Lambda''$ and $\modu \Lambda'$ respectively. Otherwise, by Lemma \ref{lem-rec}(1) and (3), $0=i^{!}j_{*}(S'')=i^{!}i_{*}(S')\cong S'$. Thus $|\Lambda'|+|\Lambda''|\leq|\Lambda|$.

Conversely, let $S$ be a simple module in $\modu \Lambda$. Since $i^{!}$ is exact by assumption, by Lemma \ref{lem-rec}(5), there is an exact sequence
$$\xymatrix@C=15pt{0\ar[r]&i_{*}i^{!}(S)\ar[r]&S\ar[r]&j_{*}j^{*}(S)\ar[r]&0}$$
in $\modu \Lambda$. It follows that either $S\cong j_{*}j^{*}(S)$ or $S\cong i_{*}i^{!}(S)$. If $S\cong i_{*}i^{!}(S)$, then by Lemma \ref{lem-fullf},  $i^{!}(S)$ is indecomposable. Consider the following exact sequence
 $$\xymatrix@C=15pt{0\ar[r]&\rad i^{!}(S)\ar[r]&i^{!}(S)\ar[r]&i^{!}(S)/\rad i^{!}(S)\ar[r]&0}$$
in $\modu \Lambda'$. Since $i_{*}$ is exact by Lemma \ref{lem-rec}(2), applying $i_{*}$ to the above exact sequence yields an exact sequence
$$\xymatrix@C=15pt{0\ar[r]&i_{*}(\rad i^{!}(S))\ar[r]&S(\cong i_{*}i^{!}(S))\ar[r]&i_{*}(i^{!}(S)/\rad i^{!}(S))\ar[r]&0}$$
in $\modu \Lambda$.
It follows that either $i_{*}(\rad i^{!}(S))=0$ or $i_{*}(i^{!}(S)/\rad i^{!}(S))=0$.
 If $i_{*}(i^{!}(S)/\rad i^{!}(S))=0$, since $i_{*}$ is fully faithful,  $i^{!}(S)/\rad i^{!}(S)=0$, which is a contradiction. If $i_{*}(\rad i^{!}(S))=0$, then $\rad i^{!}(S)=0$, and so $i^{!}(S)$ is semisimple.
Notice that $i^{!}(S)$ is indecomposable, so $i^{!}(S)$ is simple.
The case for $S\cong j_{*}j^{*}(S)$ is similar. Thus $|\Lambda|\leq |\Lambda'|+|\Lambda''|$.
Therefore, $|\Lambda|= |\Lambda'|+|\Lambda''|$.
\end{proof}

\section{Support $\tau$-tilting modules in a recollement}\label{tor}
In this section, we mainly want to study how to construct support $\tau$-tilting modules in a recollement. By  Adachi-Iyama-Reiten's correspondence given in Remark \ref{bijection}, we only need to discuss the case of functorially finite torsion classes and functorially finite torsionfree classes. Thus the first step is to give the construction of torsion classes in a recollement of abelian categories, which has done by Ma and Huang, that is,

\begin{lemma}{\rm (\cite[Theorems 1 and 2]{MXHZY})}\label{3.1}
Let $(\mathcal{A},\mathcal{B},\mathcal{C})$ be a recollement of abelian categories, and let $(\mathcal{T}',\mathcal{F}')$ and $(\mathcal{T}'',\mathcal{F}'')$ be torsion pairs in $\mathcal{A}$ and $\mathcal{C}$
respectively. There is a torsion pair $(\mathcal{T},\mathcal{F})$ in $\mathcal{B}$ defined by
\begin{align*}
\mathcal{T}&:=\{B\in\mathcal{B}\mid i^{*}(B)\in\mathcal{T}'\ \text{and}\ j^{*}(B)\in\mathcal{T}''\},\\
\mathcal{F}&:=\{B\in\mathcal{B}\mid i^{!}(B)\in\mathcal{F}'\ \text{and}\ j^{*}(B)\in\mathcal{F}''\}.
\end{align*}
In this case, we say that $(\mathcal{T},\mathcal{F})$ is a {glued torsion pair} with respect to $(\mathcal{T'},\mathcal{F'})$ and $(\mathcal{T''},\mathcal{F''})$.

On the other hand, let $(\mathcal{T},\mathcal{F})$ be a torsion pair in $\mathcal{B}$. Then
\begin{itemize}
\item[(1)] $(i^{*}(\mathcal{T}),i^{!}(\mathcal{F}))$ is a torsion pair in $\mathcal{A}$.
\item[(2)] $j_{*}j^{*}(\mathcal{F})\subseteq\mathcal{F}$ if and only if $(j^{*}(\mathcal{T}),j^{*}(\mathcal{F}))$ is a torsion pair in $\mathcal{C}$.
\end{itemize}
\end{lemma}

Now the second step is to study the  functorially finite property for these torsion classes and torsionfree classes. Before doing it,
we need the following lemma.

\begin{lemma}\label{lem-3.1}
Let $(\mathcal{A},\mathcal{B},\mathcal{C})$ be a recollement of abelian categories, and let $(\mathcal{T},\mathcal{F})$ be a torsion pair in $\mathcal{B}$. Then
\begin{itemize}
\item[(1)] $i_{*}i^{!}(\mathcal{T})\subseteq\mathcal{T}$ if and only if $i^{*}(\mathcal{T})=i^{!}(\mathcal{T})$.
\item[(2)] $i_{*}i^{*}(\mathcal{F})\subseteq\mathcal{F}$ if and only if $i^{*}(\mathcal{F})=i^{!}(\mathcal{F})$.
\end{itemize}
\end{lemma}
\begin{proof}
(1) For the only if part. Let $X\in \mathcal{T}$ and $Y\in\mathcal{F}$. By Lemma \ref{lem-rec}(4), there is an exact sequence
$$\xymatrix@C=15pt{X\ar[r]&i_{*}i^{*}(X)\ar[r]&0}$$
in $\mathcal{B}$.
Applying $\Hom_{\mathcal{B}}(-,Y)$ to the above exact sequence
yields an exact sequence
$$\xymatrix@C=15pt{0\ar[r]&\Hom_{\mathcal{B}}(i_{*}i^{*}(X),Y)
\ar[r]&\Hom_{\mathcal{B}}(X,Y)}.$$
It follows that $\Hom_{\mathcal{B}}(i_{*}i^{*}(X),Y)=0$ since $\Hom_{\mathcal{B}}(X,Y)=0$. So $i_{*}i^{*}(X)\in {^{\perp_{0}}\mathcal{F}}=\mathcal{T}$ and $i_{*}i^{*}(\mathcal{T})\subseteq\mathcal{T}$.
Note that $i^{!}i_{*}\cong 1_{\mathcal{A}}$ by Lemma \ref{lem-rec}(3), so $i^{*}(\mathcal{T})\subseteq i^{!}(\mathcal{T})$.

Since $i_{*}i^{!}(\mathcal{T})\subseteq\mathcal{T}$ and $i^{*}i_{*}\cong 1_{\mathcal{A}}$ by assumption and Lemma \ref{lem-rec}(3), it follows that $i^{!}(\mathcal{T})\subseteq i^{*}(\mathcal{T})$.
Thus $i^{*}(\mathcal{T})= i^{!}(\mathcal{T})$.

The if part is obvious.

(2) It is similar to (1).
\end{proof}

Next we can show that, under certain conditions, the glued  torsion class (resp.  torsionfree class) is  functorially finite if the original  torsion classes (resp.  torsionfree classes) are functorially finite.

\begin{theorem}\label{rec-func}
Let $(\mathcal{A},\mathcal{B},\mathcal{C})$ be a recollement of abelian categories, and let $(\mathcal{T}',\mathcal{F}')$ and $(\mathcal{T}'',\mathcal{F}'')$ be torsion pairs in $\mathcal{A}$ and $\mathcal{C}$
respectively, and $(\mathcal{T},\mathcal{F})$ a glued torsion pair in $\mathcal{B}$ with respect to $(\mathcal{T}',\mathcal{F}')$ and $(\mathcal{T}'',\mathcal{F}'')$.
\begin{itemize}
\item[(1)] Assume that $i^{!}$ is exact and $i_{*}i^{!}(\mathcal{T})\subseteq\mathcal{T}$. If $\mathcal{T'}$ and $\mathcal{T''}$ are functorially finite, then $\mathcal{T}$ is functorially finite.
    \item[(2)] Assume that $i^{*}$ is exact and $i_{*}i^{*}(\mathcal{F})\subseteq\mathcal{F}$. If $\mathcal{F}'$ and $\mathcal{F''}$ are functorially finite, then $\mathcal{F}$ is functorially finite.
        \end{itemize}
\end{theorem}

\begin{proof} We only prove the assertion (1), the proof of (2) is similar. First of all, since $(\mathcal{T},\mathcal{F})$ is a torsion pair in $\mathcal{B}$, $\mathcal{T}$ is contravariantly finite, so we only need to prove that $\mathcal{T}$ is covariantly finite.
Assume that $\mathcal{T}'$ and $\mathcal{T}''$ are functorially finite, and
let $B\in\mathcal{B}$. By Lemma \ref{lem-rec}(4) and (5),
there exist exact sequences
$$\xymatrix@C=15pt{0\ar[rr]&&i_{*}(A)\ar[rr]&&j_{!}j^{*}(B)
\ar[rr]^{\epsilon_{B}}\ar@{>>}[rd]_-{\epsilon_{B}'}&&
B\ar[rr]&&i_{*}i^{*}(B)\ar[rr]&&0,\\
&&&&&\Ima \epsilon_{B}\ar@{>->}[ru]_{i_{B}}}$$
$$\xymatrix@C=15pt{0\ar[r]&i_{*}i^{!}(B)\ar[r]^-{\lambda_{B}}
&B\ar[r]^-{\eta_{B}}&j_{*}j^{*}(B)\ar[r]&0}$$
in $\mathcal{B}$ with $A\in\mathcal{A}$.

Since $\mathcal{T}''$ is covariantly finite and $j^{*}(B)\in \mathcal{C}$,
there exists a left $\mathcal{T''}$-approximation of $j^{*}(B)$ in $\mathcal{C}$ as follows: $\xymatrix@C=15pt{j^{*}(B)\ar[r]^-{f}&X''}$ with $X''\in\mathcal{T}''$.
Since $j_{!}$ is right exact by Lemma \ref{lem-rec}(2), we get the following exact sequence
\begin{align}\label{j_{!}}
\xymatrix@C=15pt{j_{!}j^{*}(B)\ar[r]^{j_{!}(f)}&j_{!}(X'')\ar[r]&j_{!}(\Coker f)\ar[r]&0}
\end{align}
in $\mathcal{B}$.
Consider the following pushout diagram
\begin{align}\label{3.6}
\xymatrix@C=15pt{
j_{!}j^{*}(B)\ar[r]^{\epsilon_{B}'}\ar[d]^{j_{!}(f)}
&\Ima \epsilon_{B}\ar@{-->}[d]^{f'}\ar[r]&0\\
j_{!}(X'')\ar[d]\ar@{-->}[r]&U\ar@{-->}[d]\ar@{-->}[r]&0\\
j_{!}(\Coker f)\ar[d]\ar@{==}[r]&\ar@{-->}[d]j_{!}(\Coker f)\\
0&0.}\end{align}
Then we get the following pushout diagram
\begin{align}\label{3.7}
\xymatrix@C=15pt{
0\ar[r]&\Ima \epsilon_{B}\ar[r]^{i_{B}}\ar[d]^{f'}&\ar[r]B\ar@{-->}[d]^{f''}&i_{*}i^{*}(B)\ar@{==}[d]\ar[r]&0\\
0\ar@{-->}[r]&U\ar[d]\ar@{-->}[r]&V''\ar@{-->}[r]\ar@{-->}[d]&i_{*}i^{*}(B)\ar@{-->}[r]&0\\
&j_{!}(\Coker f)\ar[d]\ar@{==}[r]&\ar@{-->}[d]j_{!}(\Coker f)\\
&0&0.}\end{align}

On the other hand, since $\mathcal{T}'$ is covariantly finite and $i^{!}(B)\in \mathcal{A}$, there exists a left $\mathcal{T'}$-approximation of $i^{!}(B)$ in $\mathcal{A}$ as follows:
$\xymatrix@C=15pt{i^{!}(B)\ar[r]^-{g}&X'}$ with $X'\in\mathcal{T}'$. Since $i_{*}$ is exact by Lemma \ref{lem-rec}(2),
we get the following exact sequence
\begin{align}\label{i_{*}}
\xymatrix@C=15pt{i_{*}i^{!}(B)\ar[r]^{i_{*}(g)}&i_{*}(X')\ar[r]&i_{*}(\Coker g)\ar[r]&0}
\end{align}
in $\mathcal{B}$.
Consider the following pushout diagram
\begin{align}\label{-2}
\xymatrix@C=15pt{
0\ar[r]&i_{*}i^{!}(B)\ar[r]^{\lambda_{B}}\ar[d]^{i_{*}(g)}&\ar[r]B\ar@{-->}[d]^{g'}&j_{*}j^{*}(B)\ar@{==}[d]\ar[r]&0\\
0\ar@{-->}[r]&i_{*}(X')\ar[d]\ar@{-->}[r]&V'\ar@{-->}[r]\ar@{-->}[d]&j_{*}j^{*}(B)\ar@{-->}[r]&0\\
&i_{*}(\Coker g)\ar[d]\ar@{==}[r]&\ar@{-->}[d]i_{*}(\Coker g)\\
&0&0.}
\end{align}
Then we get the following pushout diagram
\begin{align}\label{-1}
\xymatrix@C=15pt{
B\ar[r]^{f''}\ar[d]^{g'}&\ar[r]V''\ar@{-->}[d]^{g''}&j_{!}(\Coker f)\ar@{==}[d]\ar[r]&0\\
V'\ar[d]\ar@{-->}[r]&X\ar@{-->}[r]\ar@{-->}[d]&j_{!}(\Coker f)\ar@{-->}[r]&0\\
i_{*}(\Coker g)\ar[d]\ar@{==}[r]&\ar@{-->}[d]i_{*}(\Coker g)\\
0&0.}
\end{align}

Since $j^{*}$ is exact (by Lemma \ref{lem-rec}(2)) and $\Ima i_{*}=\Ker j^{*}$, applying $j^{*}$ to the middle column in the diagram (\ref{-1}) yields an exact sequence $j^{*}(V'')\to j^{*}(X)\to 0$; applying $j^{*}$ to the middle row in diagram (\ref{3.7}) yields an exact sequence
$j^{*}(U)\to j^{*}(V'')\to 0$; applying $ j^{*}$ to the middle row in the diagram (\ref{3.6}) yields an exact sequence
$\xymatrix@C=15pt{j^{*}j_{!}(X'')\ar[r]&j^{*}(U)\ar[r]&0}$.
So we have an exact sequence $\xymatrix@C=15pt{j^{*}j_{!}(X'')\ar[r]&j^{*}(X)\ar[r]&0}$.
Notice that $\mathcal{T''}$ is closed under quotient objects and $j^{*}j_{!}(X'')\cong X''\in\mathcal{T''}$, thus $j^{*}(X)\in\mathcal{T''}$.

Since $i^{!}$ is exact by assumption, $i^{*}j_{*}=0$ by Lemma \ref{lem-rec}(6).
Since $i^{*}$ is right exact by Lemma \ref{lem-rec}(2),
applying $i^{*}$ to the middle row in the diagram (\ref{-2}) yields that $\xymatrix@C=15pt{i^{*}i_{*}(X')\ar[r]&i^{*}(V')\ar[r]&0}$ is exact.
Since $i^{*}j_{!}=0$ by Lemma \ref{lem-rec}(1),
applying $i^{*}$ to the middle row in the diagram (\ref{-1}) yields that
$\xymatrix@C=15pt{i^{*}(V')\ar[r]&i^{*}(X)\ar[r]&0}$ is exact.
Thus we have an exact sequence $\xymatrix@C=15pt{i^{*}i_{*}(X')\ar[r]&i^{*}(X)\ar[r]&0}$.
Notice that $\mathcal{T'}$ is closed under quotient objects and $i^{*}i_{*}(X')\cong X'\in \mathcal{T'}$, so $i^{*}(X)\in \mathcal{T'}$. Thus
$X\in \mathcal{T}$.

Now we claim that $\xymatrix@C=15pt{g''f'': B\ar[r]&X}$ is a left $\mathcal{T}$-approximation of $B$ in $\mathcal{B}$.

Let $\widetilde{X}\in\mathcal{T}$ and $\xymatrix@C=15pt{h: B\ar[r]&\widetilde{X}}$ be any morphism in $\mathcal{B}$.
Applying the functor $\Hom_{\mathcal{B}}(-,\widetilde{X})$ to the exact sequence (\ref{j_{!}}) yields the following commutative diagram
$${\small\xymatrix@C=18pt{
0\ar[r]&\Hom_{\mathcal{B}}(j_{!}(\Coker f),\widetilde{X})
\ar[r]\ar[d]^{\cong}&\Hom_{\mathcal{B}}(j_{!}(X''),\widetilde{X})
\ar[rr]^-{\Hom_{\mathcal{B}}(j_{!}(f),\widetilde{X})}\ar[d]^{\cong}&&\Hom_{\mathcal{B}}(j_{!}j^{*}(B),\widetilde{X})
\ar[d]^{\cong}\\
0\ar[r]&\Hom_{\mathcal{C}}(\Coker f,j^{*}(\widetilde{X}))\ar[r]&\Hom_{\mathcal{C}}(X'',j^{*}(\widetilde{X}))
\ar[rr]^-{\Hom_{\mathcal{C}}(f,j^{*}(\widetilde{X}))}&&\Hom_{\mathcal{C}}(j^{*}(B),j^{*}(\widetilde{X})).
}}$$
Notice that $\Hom_{\mathcal{C}}(f,j^{*}(\widetilde{X}))$ is epic since $f$ is a left $\mathcal{T''}$-approximation of $j^{*}(B)$ and $j^{*}(\widetilde{X})\in\mathcal{T''}$,
so $\Hom_{\mathcal{B}}(j_{!}(f),\widetilde{X})$ is epic. It follows that there exists a morphism $\xymatrix@C=15pt{\alpha :j_{!}(X'')\ar[r]&\widetilde{X}}$ such that $hi_{B}\epsilon_{B}'=\alpha j_{!}(f)$. Since the diagram (\ref{3.6}) is a pushout, there exists a morphism $\xymatrix@C=15pt{\beta: U\ar[r]&\widetilde{X}}$ such that $hi_{B}=\beta f'$. Since the diagram (\ref{3.7}) is also a pushout, there is a morphism $\xymatrix@C=15pt{\beta'': V''\ar[r]&\widetilde{X}}$ such that $h=\beta'' f''$.

Applying $\Hom_{\mathcal{B}}(-,\widetilde{X})$ to the exact sequence (\ref{i_{*}}) yields the following commutative diagram
$${\small\xymatrix@C=18pt{
0\ar[r]&\Hom_{\mathcal{B}}(i_{*}(\Coker g),\widetilde{X})
\ar[r]\ar[d]^{\cong}&\Hom_{\mathcal{B}}(i_{*}(X'),\widetilde{X})
\ar[rr]^-{\Hom_{\mathcal{B}}(i_{*}(g),\widetilde{X})}\ar[d]^{\cong}&&
\Hom_{\mathcal{B}}(i_{*}i^{!}(B),\widetilde{X})\ar[d]^{\cong}\\
0\ar[r]&\Hom_{\mathcal{A}}(\Coker g,i^{!}(\widetilde{X}))
\ar[r]&\Hom_{\mathcal{A}}(X',i^{!}(\widetilde{X}))
\ar[rr]^-{\Hom_{\mathcal{A}}(g,i^{!}(\widetilde{X}))}&&\Hom_{\mathcal{A}}(i^{!}(B),i^{!}(\widetilde{X})).
}}$$
Since $i_{*}i^{!}(\mathcal{T})\subseteq\mathcal{T}$ by assumption, $i^{*}(\mathcal{T})=i^{!}(\mathcal{T})$ by Lemma \ref{lem-3.1}(1). Notice that ${\Hom_{\mathcal{A}}(g,i^{!}(\widetilde{X}))}$ is epic since $g$ is a left $\mathcal{X'}$-approximation and $i^{!}(\widetilde{X})\in i^{!}(\mathcal{T})=i^{*}(\mathcal{T})\in \mathcal{T'}$,
so $\Hom_{\mathcal{B}}(i_{*}(g),\widetilde{X})$ is epic and there exists a morphism $\xymatrix@C=15pt{\alpha': i_{*}(X')\ar[r]&\widetilde{X}}$ such that $h\lambda_{B}=\alpha'i_{*}(g)$. Notice that the diagram (\ref{-2}) is a pushout, so there exists a morphism $\xymatrix@C=15pt{\beta': V'\ar[r]&\widetilde{X}}$ such that $h=\beta'g'$. Thus $\beta''f''=h=\beta'g'$. Since the diagram (\ref{-1}) is a pushout, there exists a morphism $\xymatrix@C=15pt{h': X\ar[r]&\widetilde{X}}$ such that $\beta''=h'g''$.
Thus $h=\beta''f''=h'g''f''$, and so $g''f''$ is a left $\mathcal{T}$-approximation of $B$. Therefore, $\mathcal{T}$ is functorially finite.
\end{proof}

Following this theorem and  Adachi-Iyama-Reiten's correspondence, we obtain

\begin{corollary}\label{cor-st}
 Let $(\modu \Lambda', \modu \Lambda, \modu \Lambda'')$ be a recollement of abelian categories, and let $T'$ and $T''$ be support $\tau$-tilting modules in $\modu \Lambda'$ and $\modu \Lambda''$ respectively, and $(\mathcal{T},\mathcal{F})$ a glued torsion pair in $\modu \Lambda$ with respect to $(\Gen T', T'^{\perp_{0}})$ and $(\Gen T'', T''^{\perp_{0}})$.
  \begin{itemize}
\item[(1)] If $i^{!}$ is exact and $i_{*}i^{!}(\mathcal{T})\subseteq\mathcal{T}$, then there is a support $\tau$-tilting $\Lambda$-module $T$ such that $(\mathcal{T},\mathcal{F})=(\Gen T, T^{\perp_{0}})$.
    \item[(2)] If $i^{*}$ is exact and $i_{*}i^{*}(\mathcal{F})\subseteq\mathcal{F}$, then there is a support $\tau$-tilting $\Lambda$-module $T$ such that $(\mathcal{T},\mathcal{F})=(\Gen T, T^{\perp_{0}})$.
  \end{itemize}
\end{corollary}

\begin{proof}
Apply Remark \ref{bijection}(1) and Theorem \ref{rec-func}.
\end{proof}

In the above, we have shown how to glue functorially finite torsion classes (resp. functorially finite torsionfree classes) from the edges to the middle in a recollement. Next, we show how to construct functorially finite torsion classes from the middle to the edges in a recollement.

We need the following two results.

\begin{lemma}\label{covariant}
Let $\xymatrix@C=15pt{j^{*}:\mathcal{B}\ar[r]&\mathcal{C}}$ be an additive functor between abelian categories $\mathcal{B}$ and $\mathcal{C}$. If $j^{*}$ has a left adjoint $j_{!}$ and $\mathcal{T}$ is a covariantly finite subcategory of $\mathcal{B}$, then $j^{*}(\mathcal{T})$ is covariantly finite in $\mathcal{C}$.
\end{lemma}
\begin{proof}
Let $C\in \mathcal{C}$, then $j_{!}(C)\in \mathcal{B}$. Since $\mathcal{T}$ is covariantly finite in $\mathcal{B}$, there is a left $\mathcal{T}$-approximation $\xymatrix@C=15pt{j_{!}(C)\ar[r]^-{f}&T}$ of $j_{!}(C)$ in $\mathcal{B}$ with $T\in\mathcal{T}$.

Let $T_{0}$ be any object in $\mathcal{T}$ and $\xymatrix@C=15pt{C\ar[r]^-{g}&j^{*}(T_{0})}$ be any morphism in $\mathcal{B}$.
Since $f$ is a left $\mathcal{T}$-approximation of $j_{!}(C)$, there exists a morphism $\xymatrix@C=15pt{\beta: T\ar[r]&T_{0}}$ such that the following diagram is commutative
$$\xymatrix@C=20pt{j_{!}(C)\ar[d]_{j_{!}(g)}\ar[r]^{f}&T\ar@{.>}[d]^{\beta}\\
j_{!}j^{*}(T_{0})\ar[r]^-{\epsilon_{T_{0}}}&T_{0}.}$$

Consider the following commutative diagram
$$\xymatrix@C=15pt{C\ar[d]_{g}\ar[rr]^-{\gamma_{C}}&&
j^{*}j_{!}(C)
\ar[d]^{j^{*}j_{!}(g)}
\\
j^{*}(T_{0})\ar[rr]^-{\gamma_{j^{*}(T_{0})}}&&j^{*}j_{!}j^{*}(T_{0})
,}$$
where $\xymatrix@C=15pt{\gamma: 1_{\mathcal{C}}\ar[r]&j^{*}j_{!}}$ is the unit of the adjoint pair $(j_{!},j^{*})$.

Then $j^{*}(\beta)j^{*}(f)\gamma_{C}=j^{*}(\epsilon_{T_{0}})
j^{*}j_{!}(g)\gamma_{C}=j^{*}(\epsilon_{T_{0}})\gamma_{j^{*}(T_{0})}g$. Since $(j_{!},j^{*})$ is an adjoint pair,
$1_{j^{*}(T_{0})}=j^{*}(\epsilon_{T_{0}})\gamma_{j^{*}(T_{0})}$, and hence $j^{*}(\beta)j^{*}(f)\gamma_{C}=g$, that is,
the following diagram is commutative
$$\xymatrix@C=40pt{C\ar[d]_{g}\ar[r]^{j^{*}(f)\gamma_{C}}&j^{*}(T)\ar@{.>}[ld]^{j^{*}(\beta)}\\
j^{*}(T_{0}).}$$
Thus $j^{*}(f)\gamma_{C}$ is a left $j^{*}(\mathcal{T})$-approximation of $C$ and $j^{*}(\mathcal{T})$ is covariantly finite in $\mathcal{C}$.

\end{proof}

Dually, we have
\begin{lemma}\label{contravariant}
Let $\xymatrix@C=15pt{j^{*}:\mathcal{B}\ar[r]&\mathcal{C}}$ be an additive functor between abelian categories $\mathcal{B}$ and $\mathcal{C}$. If $j^{*}$ has a right adjoint $j_{*}$ and $\mathcal{T}$ is a contravariantly finite subcategory of $\mathcal{B}$, then $j^{*}(\mathcal{T})$ is contravariantly finite in $\mathcal{C}$.
\end{lemma}

Now we can obtain the following result.

\begin{theorem}\label{con-func}
Let $(\mathcal{A},\mathcal{B},\mathcal{C})$ be a recollement of abelian categories, and let $(\mathcal{T},\mathcal{F})$ be a torsion pair in $\mathcal{B}$.
\begin{itemize}
\item[(1)] Assume that $\mathcal{T}$ is functorially finite.
\begin{itemize}
\item[(a)] If $i^{*}$ has a left adjoint, then $(i^{*}(\mathcal{T}),i^{!}(\mathcal{F}))$ is a torsion pair with $i^{*}(\mathcal{T})$ functorially finite in $\mathcal{A}$.
    \item[(b)] If $j_{*}j^{*}(\mathcal{F})\subseteq\mathcal{F}$, then $(j^{*}(\mathcal{T}),j^{*}(\mathcal{F}))$
is a torsion pair with $j^{*}(\mathcal{T})$ functorially finite in $\mathcal{C}$.
    \end{itemize}
\item[(2)] Assume that $\mathcal{F}$ is functorially finite.
\begin{itemize}
\item[(a)] If $i^{!}$ has a right adjoint, then $(i^{*}(\mathcal{T}),i^{!}(\mathcal{F}))$ is a torsion pair with $i^{!}(\mathcal{F})$ functorially finite in $\mathcal{A}$.
\item[(b)] If $j_{*}j^{*}(\mathcal{F})\subseteq\mathcal{F}$, then $(j^{*}(\mathcal{T}),j^{*}(\mathcal{F}))$
is a torsion pair with $j^{*}(\mathcal{F})$ functorially finite in $\mathcal{C}$.
\end{itemize}
\end{itemize}
\end{theorem}
\begin{proof}
Apply Lemmas \ref{3.1}, \ref{covariant} and \ref{contravariant}.
\end{proof}

Further, we have

\begin{corollary}\label{con-cor-st}
Let $(\modu \Lambda', \modu \Lambda, \modu \Lambda'')$ be a recollement of module categories, and let $T$ be a support $\tau$-tilting $\Lambda$-module and $(\mathcal{T},\mathcal{F}):=(\Gen T,T^{\perp_{0}})$ a torsion pair induced by $T$ in $\modu \Lambda$.
\begin{itemize}
\item[(1)] If $i^{*}$ has a left adjoint  or $i^{!}$ has a right adjoint, then there is a support $\tau$-tilting module $T'$ in $\modu \Lambda'$ such that $(\Gen T',T'^{\perp_{0}})=(i^{*}(\mathcal{T}),i^{!}(\mathcal{F}))$.
\item[(2)]
\begin{itemize}
\item[(a)] If $j_{*}j^{*}(\mathcal{F})\subseteq\mathcal{F}$, then there is a support $\tau$-tilting module $T''$ in $\modu \Lambda''$ such that $(\Gen T'',T''^{\perp_{0}})=(j^{*}(\mathcal{T}),j^{*}(\mathcal{F}))$.
\item[(b)] If $j_{*}$ is exact and $j_{*}j^{*}(\mathcal{T})\subseteq\mathcal{T}$, then there is a support $\tau$-tilting module $T''$ in $\modu \Lambda''$ with respect to $j^{*}(\mathcal{T})$.
     Moreover,  $j_{*}j^{*}(\mathcal{F})\subseteq\mathcal{F}$ if and only if $\Gen T''=j^{*}(\mathcal{T})$ and $T''^{\perp_{0}}=j^{*}(\mathcal{F})$.
\item[(c)] If $j_{!}$ is exact and $j_{!}j^{*}(\mathcal{F})\subseteq\mathcal{F}$, then there is a support $\tau$-tilting module $T''$ in $\modu \Lambda''$ with respect to $j^{*}(\mathcal{F})$. Moreover,  $j_{*}j^{*}(\mathcal{F})\subseteq\mathcal{F}$ if and only if $\Gen T''=j^{*}(\mathcal{T})$ and $T''^{\perp_{0}}=j^{*}(\mathcal{F})$.
\end{itemize}
\end{itemize}
\end{corollary}
\begin{proof}
By assumption and Remark \ref{bijection}(1),  $\mathcal{T}$ and $\mathcal{F}$ are functorially finite in $\modu \Lambda$.
The assertions (1) and (2)(a) follow from Remark \ref{bijection}(1) and Theorem \ref{con-func}.

(2)(b) Let
$$\xymatrix@C=15pt{0\ar[r]&X\ar[r]&Y\ar[r]&Z\ar[r]&0}$$
be an exact sequence in $\modu \Lambda''$.
Since $j_{*}$ is exact by assumption, applying $j_{*}$ to the above exact sequence yields an exact sequence
$$\xymatrix@C=15pt{0\ar[r]&j_{*}(X)\ar[r]&j_{*}(Y)\ar[r]&j_{*}(Z)\ar[r]&0}$$
in $\modu \Lambda$.
Assume that $X$, $Z\in j^{*}(\mathcal{T})$, since $j_{*}j^{*}(\mathcal{T})\subseteq\mathcal{T}$, $j_{*}(X)$ and $j_{*}(Z)$ are in $\mathcal{T}$. It follows that $j_{*}(Y)\in \mathcal{T}$ since $\mathcal{T}$ is closed under extensions.
Since $j^{*}j_{*}\cong 1_{\modu \Lambda''}$ by Lemma \ref{lem-rec}(3), $Y\cong j^{*}j_{*}(Y)\in j^{*}(\mathcal{T})$. So $j^{*}(\mathcal{T})$ is closed under extensions.
Assume that $Y\in j^{*}(\mathcal{T})$, since $j_{*}j^{*}(\mathcal{T})\subseteq\mathcal{T}$,  $j_{*}(Y)\in \mathcal{T}$.
It follows that $j_{*}(Z)\in\mathcal{T}$ since $\mathcal{T}$ is closed under quotient modules, and since $j^{*}j_{*}\cong 1_{\modu \Lambda''}$ by Lemma \ref{lem-rec}(3), $Z\cong j^{*}j_{*}(Z)\in j^{*}(\mathcal{T})$. So $j^{*}(\mathcal{T})$ is closed under quotient modules. Thus $j^{*}(\mathcal{T})$ is a torsion class in $\modu \Lambda''$.
Moreover, notice that $\mathcal{T}$ is functorially finite, so by Lemmas \ref{covariant} and \ref{contravariant}, $j^{*}(\mathcal{T})$ is a functorially finite torsion class in $\modu \Lambda''$.
By Remark \ref{bijection}(1), there is a support $\tau$-tilting module $T''$ in $\modu \Lambda''$ with respect to $j^{*}(\mathcal{T})$.

Moreover, $T''$ induces a torsion pair $(\Gen T'',T''^{\perp_{0}})$ in $\modu \Lambda''$.
Again by Remark \ref{bijection}(1), $\Gen T''=j^{*}(T)$.
Furthermore, since $j_{*}j^{*}(\mathcal{F})\subseteq \mathcal{F}$ by assumption,
$(j^{*}(\mathcal{T}),j^{*}(\mathcal{F}))$ is a torsion pair in $\modu \Lambda''$ by Lemma \ref{3.1}, and hence $T''^{\perp_{0}}= j^{*}(\mathcal{F})$.

Conversely, assume that $\Gen T''= j^{*}(\mathcal{T})$ and $T''^{\perp_{0}}= j^{*}(\mathcal{F})$.
Since $(\Gen T'',T''^{\perp_{0}})$ is a torsion pair induced by $T''$,  $(j^{*}(\mathcal{T}),j^{*}(\mathcal{F}))$ is a torsion pair.
The assertion follows from Lemma \ref{3.1}.

(2)(c) It is similar to (2)(b).
\end{proof}

In particular, we can give an explicit construction for these support $\tau$-tilting modules in $\modu \Lambda''$ respectively under certain conditions.
\begin{proposition}\label{con-prop-st}
Let $(\modu \Lambda', \modu \Lambda, \modu \Lambda'')$ be a recollement of module categories, and let $T$ be a support $\tau$-tilting $\Lambda$-module and $(\mathcal{T},\mathcal{F}):=(\Gen T,T^{\perp_{0}})$ a torsion pair induced by $T$ in $\modu \Lambda$.
If $j_{!}$, $j_{*}$ are exact, and $j_{*}j^{*}(\mathcal{T})\subseteq\mathcal{T}$, then $j^{*}(T)$ is a support $\tau$-tilting $\Lambda''$-module. Moreover,  $j_{*}j^{*}(\mathcal{F})\subseteq\mathcal{F}$ if and only if $\Gen j^{*}(T)=j^{*}(\mathcal{T})$ and $(j^{*}(T))^{\perp_{0}}=j^{*}(\mathcal{F})$.
\end{proposition}
\begin{proof}
By assumption and Corollary \ref{con-cor-st}(2)(b), there is a support $\tau$-tilting module in $\modu \Lambda''$ with respect to $j^{*}(\mathcal{T})$.
Since $j_{*}$ is exact and $j_{*}j^{*}(\mathcal{T})\subseteq\mathcal{T}$ by assumption, $\Ext_{\Lambda''}^{1}(j^{*}(T),j^{*}(\mathcal{T}))\cong \Ext_{\Lambda}^{1}(T,j_{*}j^{*}(\mathcal{T}))=0$ by Proposition \ref{prop-adjoint}(4), so $j^{*}(T)$ is $\Ext$-projective in $j^{*}(\mathcal{T})$. Let $T''$ be an indecomposable $\Ext$-projective object in $j^{*}(\mathcal{T})$. Notice that $j_{!}$ is exact by assumption, so $\Ext_{\Lambda}^{1}(j_{!}(T''),\mathcal{T})\cong \Ext_{\Lambda''}^{1}(T'',j^{*}(\mathcal{T}))=0$ by Proposition \ref{prop-adjoint}(3), and hence $j_{!}(T'')$ is $\Ext$-projective in $\mathcal{T}$. It follows that $j_{!}(T'')\in\add P(\mathcal{T})=\add T$ and $T''\cong j^{*}j_{!}(T'')\in \add j^{*}(T)$. Thus $j^{*}(T)$ is a support $\tau$-tilting $\Lambda$-module by Remark \ref{bijection}(1).

The second assertion follows from Corollary \ref{con-cor-st}(2)(b).
\end{proof}

\section{$\tau$-tilting modules in a recollement}

In Section 3, we have discussed the construction of support $\tau$-tilting modules in a recollement.  In this section, we further study how to construct $\tau$-tilting modules in a recollement. By  Adachi-Iyama-Reiten's correspondence given in Remark \ref{bijection}, we only need to discuss the case of sincere functorially finite torsion classes.

We need  an easy property of the composition series.

\begin{lemma}\label{lem-composition}
Let $\Lambda'$ and $\Lambda$ be artin algebras and $\xymatrix@C=15pt{F: \modu \Lambda'\ar[r]&\modu \Lambda}$ an exact functor, and let $M$ be a $\Lambda'$-module. If $F(S)$ is either simple or zero in $\modu \Lambda$ for any simple $\Lambda'$-module $S$, and if
$$\xymatrix@C+15pt{0= M_{0}\subseteq M_{1}\subseteq \cdots \subseteq M_{n-2}\subseteq M_{n-1}\subseteq M_{n}=M}$$
is a composition series of $M$,
then
$$\xymatrix@C+15pt{0= F(M_{0})\subseteq F(M_{1})\subseteq \cdots \subseteq F(M_{n-2})\subseteq F(M_{n-1})\subseteq F(M_{n})=F(M)}$$
is a submodule chain of $F(M)$ with $F(M_{i+1})/F(M_{i})$ either simple or zero in $\modu \Lambda$ for $0\leq i\leq n-1$.
\end{lemma}
\begin{proof}
By assumption, there are exact sequences
$$\xymatrix@C=15pt{0\ar[r]&M_{i}\ar[r]&M_{i+1}\ar[r]&S_{i+1}\ar[r]&0}$$
in $\modu \Lambda'$ with $S_{i+1}$ simple $\Lambda'$-modules, where $0\leq i\leq n-1$.
Since $F$ is exact, applying $F$ to the above exact sequences
yields the following exact sequences
$$\xymatrix@C=15pt{0\ar[r]&F(M_{i})\ar[r]&
F(M_{i+1})\ar[r]&F(S_{i+1})\ar[r]&0}.$$
Thus the assertion follows from the assumption that $F(S_{i+1})$ is either simple or zero in $\modu \Lambda$.
\end{proof}

From now on, assume that $(\modu \Lambda',\modu \Lambda,\modu \Lambda'')$ is a recollement of module categories:
$$\xymatrix{\modu \Lambda'\ar[rr]!R|-{i_{*}}&&\ar@<-2ex>[ll]!R|-{i^{*}}
\ar@<2ex>[ll]!R|-{i^{!}}\modu \Lambda
\ar[rr]!L|-{j^{*}}&&\ar@<-2ex>[ll]!L|-{j_{!}}\ar@<2ex>[ll]!L|-{j_{*}}
\modu \Lambda''.}$$
We can show that, under certain conditions, the glued  torsion class is sincere functorially finite if the original  torsion classes  are sincere functorially finite.

\begin{theorem}\label{sincere}
Let $(\mathcal{T}',\mathcal{F}')$ and $(\mathcal{T}'',\mathcal{F}'')$ be torsion pairs in $\modu \Lambda'$ and $\modu \Lambda''$
respectively, and $(\mathcal{T},\mathcal{F})$ a glued torsion pair with respect to $(\mathcal{T}',\mathcal{F}')$ and $(\mathcal{T}'',\mathcal{F}'')$. Assume that $\mathcal{T'}$ and $\mathcal{T''}$ are sincere functorially finite. If one of the following conditions is satisfied
\begin{itemize}
\item[(1)] $i^{!}$ is exact and $i_{*}i^{!}(\mathcal{T})\subseteq\mathcal{T}$,
    \item[(2)] $i^{*}$ is exact and $i_{*}i^{*}(\mathcal{F})\subseteq\mathcal{F}$,
    \end{itemize}
    then $\mathcal{T}$ is sincere functorially finite.
\end{theorem}
\begin{proof}
Since $\mathcal{T'}$ and $\mathcal{T''}$ are sincere by assumption, there exist objects $A'\in \mathcal{T'}$ and $A''\in \mathcal{T''}$ such that all simple modules appear as  composition factors of $A'$ and $A''$ in $\modu \Lambda'$ and $\modu \Lambda''$ respectively.

(1) By Theorem \ref{rec-func}(1),  $\mathcal{T}$ is functorially finite.
 Since $\Ima i_{*}=\Ker j^{*}$ and $i^{*}i_{*}\cong 1_{\modu \Lambda'}$ (by Lemma \ref{lem-rec}(3)), $j^{*}i_{*}(A')=0$ and $i^{*}i_{*}(A')\cong A'\in\mathcal{T'}$, so $i_{*}(A')\in\mathcal{T}$. On the other hand, since $i^{!}$ is exact,  $i^{*}j_{*}=0$ by Lemma \ref{lem-rec}(6), it follows that $i^{*}j_{*}(A'')=0$.
Notice that $j^{*}j_{*}\cong 1_{\modu \Lambda''}$ by Lemma \ref{lem-rec}(3), so $j^{*}j_{*}(A'')\cong A''\in\mathcal{T''}$, and hence $j_{*}(A'')\in \mathcal{T}$. Set $A:=i_{*}(A')\oplus j_{*}(A'')$. Then $A\in\mathcal{T}$.

Since $i^{!}$ is exact by assumption, $j_{*}$ is exact by Lemma \ref{lem-2.6}. Note that $i_{*}$ is also exact, so by Lemmas \ref{lem-composition} and \ref{lem-simple} all simple modules appear as a composition factor of $A$ in $\modu \Lambda$. Thus $\mathcal{T}$ is sincere.

(2) By Theorem \ref{rec-func}(2),  $\mathcal{F}$ is functorially finite. So by Remark \ref{bijection}(1), $\mathcal{T}$ is functorially finite. As a similar argument to (1), $A:=i_{*}(A')\oplus j_{!}(A'')\in \mathcal{T}$. Since $i^{*}$ is exact by assumption, by Lemma \ref{lem-2.6}, $j_{!}$ is exact. Note that $i_{*}$ is also exact, so by Lemmas \ref{lem-composition} and \ref{lem-simple} all simple modules appear as a composition factor of $A$ in $\modu \Lambda$. Thus $\mathcal{T}$ is sincere.
\end{proof}

Following this theorem and  Adachi-Iyama-Reiten's correspondence, we obtain the following corollary.

\begin{corollary}\label{cor-tau}
 Let $T'$ and $T''$ be $\tau$-tilting modules in $\modu \Lambda'$ and $\modu \Lambda''$ respectively. If one of the following conditions is satisfied
  \begin{itemize}
  \item[(1)] $i^{!}$ is exact and $i_{*}i^{!}(\mathcal{T})\subseteq\mathcal{T}$,
       \item[(2)] $i^{*}$ is exact and $i_{*}i^{*}(\mathcal{F})\subseteq\mathcal{F}$,
       \end{itemize}
       then there is a $\tau$-tilting module $T$ such that $(\mathcal{T},\mathcal{F})=(\Gen T,T^{\perp_{0}})$, where  $(\mathcal{T},\mathcal{F})$ is a glued torsion pair in $\modu \Lambda$ with respect to $(\Gen T', T'^{\perp_{0}})$ and $(\Gen T'', T''^{\perp_{0}})$.
 \end{corollary}

Now, we can give a construction of $\tau$-tilting modules in a recollement.

\begin{proposition}\label{rec-tautilting}
 Let $T'$ and $T''$ be $\tau$-tilting modules in $\modu \Lambda'$ and $\modu \Lambda''$ respectively. If $i^{!}$, $j_{!}$ are exact and $i_{*}i^{!}(\mathcal{T})\subseteq \mathcal{T}$, then $T:=i_{*}(T')\oplus j_{!}(T'')$ is a $\tau$-tilting $\Lambda$-module and $(\mathcal{T},\mathcal{F})=(\Gen T,T^{\perp_{0}})$, where  $(\mathcal{T},\mathcal{F})$ is a glued torsion pair in $\modu \Lambda$ with respect to $(\Gen T', T'^{\perp_{0}})$ and $(\Gen T'', T''^{\perp_{0}})$.
 \end{proposition}
\begin{proof}
By Corollary \ref{cor-tau}(1), there is a $\tau$-tilting module $\overline{T}$ such that $(\mathcal{T},\mathcal{F})=(\Gen \overline{T},\overline{T}^{\perp_{0}})$.
By Remark \ref{bijection}(1), we can take $\overline{T}=P(\mathcal{T})$.

Since $i_{*}i^{!}(\mathcal{T})\subseteq \mathcal{T}$ by assumption, by Lemma \ref{lem-3.1}(1), $i^{*}(\mathcal{T})=i^{!}(\mathcal{T})$.
Notice that $i^{!}$ and $j_{!}$ are exact by assumption, so by Proposition \ref{prop-adjoint}(2) and (3),
\begin{align*}
\Ext_{\Lambda}^{1}(T,\mathcal{T})&= \Ext_{\Lambda}^{1}(i_{*}(T')\oplus j_{!}(T''),\mathcal{T})\\
&= \Ext_{\Lambda}^{1}(i_{*}(T'),\mathcal{T})\oplus \Ext_{\Lambda}^{1}(j_{!}(T''),\mathcal{T})\\
&\cong \Ext_{\Lambda'}^{1}(T',i^{!}(\mathcal{T})) \oplus \Ext_{\Lambda''}^{1}(T'',j^{*}(\mathcal{T}))\\
&=0.
\end{align*}
It follows that $T$ is an $\Ext$-projective object in $\mathcal{T}$. Note that all indecomposable direct summands of $T$ are pairwise nonisomorphic, so $T$ is isomorphic to a direct summand of $\overline{T}$.
On the other hand,
\begin{align*}
|T|=&|i_{*}(T')|+|j_{!}(T'')|\\
=&|T''|+|T''|\text{\ \ \ \  (by Lemma \ref{lem-fullf})}\\
=&|\Lambda'|+|\Lambda''|\\
=&|\Lambda|\text{ \ \ \ \ (by Lemma \ref{lem-simple})}\\
=&|\overline{T}|.
\end{align*}
Thus $T\cong \overline{T}$ is a $\tau$-tilting module in $\modu \Lambda$.
\end{proof}

In the above, we have shown how to glue sincere functorially finite torsion classes from the edges to the middle in a recollement. Next, we will consider how to construct sincere functorially finite torsion classes from the middle to the edges in a recollement.

\begin{theorem}\label{sincere-fun}
Let $(\mathcal{T},\mathcal{F})$ be a torsion pair in $\modu \Lambda$. Assume that $\mathcal{T}$ is sincere functorially finite.
\begin{itemize}
\item[(1)] If $i^{*}$ has a left adjoint, then $(i^{*}(\mathcal{T}),i^{!}(\mathcal{F}))$ is a torsion pair with $i^{*}(\mathcal{T})$ sincere functorially finite in $\modu \Lambda'$.
    \item[(2)] Assume $i^{*}$ or $i^{!}$ is exact.
     \begin{itemize}
\item[(a)] If $j_{*}j^{*}(\mathcal{F})\subseteq \mathcal{F}$, then $(j^{*}(\mathcal{T}),j^{*}(\mathcal{F}))$ is a torsion pair with $j^{*}(\mathcal{T})$ sincere functorially finite in $\modu \Lambda''$.
\item[(b)] If $j_{*}$ is exact and $j_{*}j^{*}(\mathcal{T})\subseteq\mathcal{T}$, then $j^{*}(\mathcal{T})$ is a sincere functorially finite torsion class in $\modu \Lambda''$.
\end{itemize}
        \end{itemize}
\end{theorem}
\begin{proof}
Since $\mathcal{T}$ is sincere, there exists a $\Lambda$-module $A$ in $\mathcal{T}$ such that all simple $\Lambda$-modules appear as a composition factor of $A$.

(1) By Lemma \ref{3.1},  $(i^{*}(\mathcal{T}),i^{!}(\mathcal{F}))$ is a torsion pair in $\modu \Lambda'$. Since $i^{*}$ has a left adjoint, by Theorem \ref{con-func}, $i^{*}(\mathcal{T})$ is functorially finite. Since $i^{*}$ has a left adjoint,  $i^{*}$ is left exact. Notice that $i^{*}$ is right exact by Lemma \ref{lem-rec}(2), so $i^{*}$ is exact. Then by Lemma \ref{lem-simple},  every simple $\Lambda$-module is of the form either $i_{*}(S')$ or $j_{!}(S'')$, where $S'$ is a simple $\Lambda'$-module and $S''$ is a simple $\Lambda''$-module. Then the classes of $\{i_{*}(S'): \text{the simple }\Lambda'\text{-module }S'\}$ and $\{j_{!}(S''): \text{the simple }\Lambda''\text{-module }S''\}$ appear as a composition factor of $A$. Since $i^{*}$ is exact, and $i^{*}j_{!}(S'')=0$ (by Lemma \ref{lem-rec}(1)) and $i^{*}i_{*}(S')\cong S'$ (by Lemma \ref{lem-rec}(3)), by Lemma \ref{lem-composition}, all simple $\Lambda'$-modules appear as a composition factor of $i^{*}(A)$.

(2)(a) We only need to prove the case that $i^{*}$ is exact; the case that $i^{!}$ is exact is similar.
By assumption and Lemma \ref{3.1},  $(j^{*}(\mathcal{T}),j^{*}(\mathcal{F}))$ is a torsion pair in $\modu \Lambda''$.
By assumption and Theorem \ref{con-func}(1)(b),  $j^{*}(\mathcal{T})$ is a functorially finite torsion class.
Since $i^{*}$ is exact by assumption, by Lemma \ref{lem-simple},  every simple $\Lambda$-module is of the form either $i_{*}(S')$ or $j_{!}(S'')$, where $S'$ is a simple $\Lambda'$-module and $S''$ is a simple $\Lambda''$-module.
Then the classes of $\{i_{*}(S'): \text{the simple }\Lambda'\text{-module }S'\}$ and $\{j_{!}(S''): \text{the simple }\Lambda''\text{-module }S''\}$ appear as a composition factor of $A$.
Since $\Ima i_{*}=\Ker j^{*}$ and $j^{*}j_{!}\cong 1_{\modu \Lambda''}$ (by Lemma \ref{lem-rec}(3)),  $j^{*}i_{*}(S')=0$ and $j^{*}j_{!}(S'')\cong S''$.
Moreover, since $j^{*}$ is exact (by Lemma \ref{lem-rec}(2)), by Lemma \ref{lem-composition}, all simple $\Lambda''$-module appear as a composition factor of $j^{*}(A)$. Thus $j^{*}(\mathcal{T})$ is sincere.

(2)(b) As a similar proof to that of Corollary \ref{con-cor-st}(2)(b),  $j^{*}(\mathcal{T})$ is a functorially finite torsion class in $\modu \Lambda''$. Moreover, as a similar argument to (2)(a), we get that $j^{*}(\mathcal{T})$ is sincere.
\end{proof}

Immediately, we have the following corollary.

\begin{corollary}\label{con-cor-tau}
Let $T$ be a $\tau$-tilting $\Lambda$-module, and $(\mathcal{T},\mathcal{F}):=(\Gen T,T^{\perp_{0}})$ a torsion pair induced by $T$ in $\modu \Lambda$.
\begin{itemize}
\item[(1)] If $i^{*}$ has a left adjoint, then there is a $\tau$-tilting module $T'$ in $\modu \Lambda'$ such that $(\Gen T',T'^{\perp_{0}})=(i^{*}(\mathcal{T}),i^{!}(\mathcal{F}))$.
    \item[(2)] Assume $i^{*}$ or $i^{!}$ is exact.
     \begin{itemize}
\item[(a)] If $j_{*}j^{*}(\mathcal{F})\subseteq \mathcal{F}$, then there is a $\tau$-tilting $\Lambda''$-module $T''$ such that $(\Gen T'',T''^{\perp_{0}})=(j^{*}(\mathcal{T}),j^{*}(\mathcal{F}))$.
\item[(b)] If $j_{*}$ is exact and $j_{*}j^{*}(\mathcal{T})\subseteq\mathcal{T}$, then there is a $\tau$-tilting $\Lambda''$-module $T''$ with respect to $j^{*}(\mathcal{T})$. Moreover,  $j_{*}j^{*}(\mathcal{F})\subseteq \mathcal{F}$ if and only if $\Gen T''=j^{*}(\mathcal{T})$ and $T''^{\perp_{0}}=j^{*}(\mathcal{F})$.
\end{itemize}
        \end{itemize}
\end{corollary}
\begin{proof}
By Remark \ref{bijection}(2) and Theorem \ref{sincere-fun}, we only need to prove the last assertion.
But the proof of the last assertion is similar to that of Corollary \ref{con-cor-st}(2)(b).
\end{proof}

The following result is a special case of Proposition \ref{con-prop-st}.
\begin{proposition}\label{prop-t}
Let $T$ be a $\tau$-tilting $\Lambda$-module, and let $(\mathcal{T},\mathcal{F}):=(\Gen T,T^{\perp_{0}})$ be a torsion pair induced by $T$ in $\modu \Lambda$.
\begin{itemize}
\item[(a)] If $i^{!}$ and $j_{!}$ are exact, then $j^{*}(T)$ is a $\tau$-tilting module in $\modu \Lambda''$. Moreover,  $j_{*}j^{*}(\mathcal{F})\subseteq\mathcal{F}$ if and only if $\Gen j^{*}(T)=j^{*}(\mathcal{T})$ and $(j^{*}(T))^{\perp_{0}}=j^{*}(\mathcal{F})$.
\item[(b)] If $i^{*}$, $j_{*}$ are exact, and $j_{*}j^{*}(\mathcal{T})\subseteq\mathcal{T}$, then $j^{*}(T)$ is a $\tau$-tilting module in $\modu \Lambda''$. Moreover,  $j_{*}j^{*}(\mathcal{F})\subseteq\mathcal{F}$ if and only if $\Gen j^{*}(T)=j^{*}(\mathcal{T})$ and $(j^{*}(T))^{\perp_{0}}=j^{*}(\mathcal{F})$.
\end{itemize}
\end{proposition}
\begin{proof}

(a) Let $X\in\mathcal{T}$ and $Y\in\mathcal{F}$. Since $i^{!}$ is exact by assumption, by Lemma \ref{lem-rec}(5), there is an exact sequence
      $$\xymatrix@C=15pt{X\ar[r]&
            j_{*}j^{*}(X)\ar[r]&0}$$
            in $\modu \Lambda$.
            Applying the functor $\Hom_{\Lambda}(-,Y)$ to the above exact sequence yields an exact sequence
              $$\xymatrix@C=15pt{0\ar[r]&\Hom_{\Lambda}( j_{*}j^{*}(X),Y)\ar[r]&\Hom_{\Lambda}(X,Y)}.$$
              It follows that $\Hom_{\Lambda}( j_{*}j^{*}(X),Y)=0$ since $\Hom_{\Lambda}(X,Y)=0$. So $ j_{*}j^{*}(X)\in{^{\perp_{0}}\mathcal{F}}=\mathcal{T}$ and $ j_{*}j^{*}(\mathcal{T})\subseteq\mathcal{T}$.
 Since $i^{!}$ is exact by assumption, $j_{*}$ is exact by Lemma \ref{lem-2.6}(2). The assertions follow from a similar proof as Proposition \ref{con-prop-st}.

(b) Since $i^{*}$ is exact by assumption, $j_{!}$ is exact by Lemma \ref{lem-2.6}(1). The assertions follow from a similar proof as Proposition \ref{con-prop-st}.
\end{proof}

\section{Examples}
We give some  examples to illustrate the obtained results.

Let $\Lambda', \Lambda''$ be artin algebras and $_{\Lambda'}M_{\Lambda''}$ an $(\Lambda',\Lambda'')$-bimodule, and let $\Lambda={\Lambda'\ {M}\choose\  0\  \ \Lambda''}$ be a triangular matrix algebra.
Then any module in $\modu \Lambda$ can be uniquely written as a triple ${X\choose Y}_{f}$ with $X\in\modu \Lambda'$, $Y\in\modu \Lambda''$
and $f\in\Hom_{\Lambda'}(M\otimes_{\Lambda''}Y,X)$ (see \cite[p.76]{AMRISSO95R} for more details).

\begin{example}
{\rm Let $\Lambda'$ be a finite dimensional algebra given by the quiver $\xymatrix@C=15pt{1\ar[r]&2}$ and $\Lambda''$ be a finite dimensional algebra given by the quiver $\xymatrix@C=12pt{3\ar[r]^{\alpha}&4\ar[r]^{\beta}&5}$ with the relation $\beta\alpha=0$. Define a triangular matrix algebra $\Lambda={\Lambda'\ \Lambda'\choose \ 0\ \ \Lambda''}$, where the right $\Lambda''$-module structure on $\Lambda'$ is induced by the unique algebra surjective homomorphsim $\xymatrix@C=15pt{\Lambda''\ar[r]^{\phi}&\Lambda'}$ satisfying $\phi(e_{3})=e_{1}$, $\phi(e_{4})=e_{2}$, $\phi(e_{5})=0$.  Then $\Lambda$ is
a finite dimensional algebra given by the quiver
$$
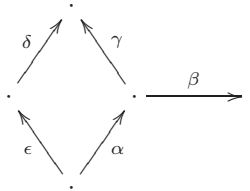
\xymatrix@C=15pt{&\cdot\\
\cdot\ar[ru]^{\delta}&&\ar[lu]_{\gamma}\cdot\ar[rr]^-{\beta}&&\cdot\\
&\ar[lu]^{\epsilon}\cdot\ar[ru]_{\alpha}}$$
with the relations $\gamma\alpha=\delta\epsilon$ and $\beta\alpha=0$. The Auslander-Reiten quiver of $\Lambda$ is
$$\xymatrix@C=15pt{{0\choose P(5)}\ar[rd]&&{S(2)\choose S(4)}\ar[rd]&&{S(1)\choose 0}\ar[rd]&&0\choose P(3)\ar[rd]\\
&{S(2)\choose P(4)}\ar[ru]\ar[rd]&&P(1)\choose S(4)\ar[ru]\ar[r]\ar[rd]&P(1)\choose P(3)\ar[r]&S(1)\choose P(3)\ar[ru]\ar[rd]&&{0\choose S(3)}.\\
S(2)\choose 0\ar[ru]\ar[rd]&&P(1)\choose P(4)\ar[ru]\ar[rd]&&0\choose S(4)\ar[ru]&&S(1)\choose S(3)\ar[ru]\\
&P(1)\choose 0\ar[ru]&&0\choose P(4)\ar[ru]}$$
By \cite[Example 2.12]{PC14H},
$$\xymatrix{\modu \Lambda'\ar[rr]!R|-{i_{*}}&&\ar@<-2ex>[ll]!R|-{i^{*}}
\ar@<2ex>[ll]!R|-{i^{!}}\modu \Lambda
\ar[rr]!L|-{j^{*}}&&\ar@<-2ex>[ll]!R|-{j_{!}}\ar@<2ex>[ll]!R|-{j_{*}}
\modu \Lambda''}$$
is a recollement of module categories, where
\begin{align*}
&i^{*}({X\choose Y}_{f})=\Coker f, & i_{*}(X)={X\choose 0},&&i^{!}({X\choose Y}_{f})=X,\\
&j_{!}(Y)={M\otimes_{\Lambda''} Y\choose Y}_{1}, & j^{*}({X\choose Y}_{f})=Y, &&j_{*}(Y)={0\choose Y}.
\end{align*}
\begin{itemize}
\item[(1)]
    Take support $\tau$-tilting modules $T'=S(1)$ and $T''=P(5)\oplus P(4)$ in $\modu \Lambda'$ and $\modu \Lambda''$ respectively. They induce torsion pairs
    \begin{align*}
    (\mathcal{T'},\mathcal{F'})=&(\add S(1),\add (P(1)\oplus S(2))),\\
    (\mathcal{T''},\mathcal{F''})=&(\add (P(5)\oplus P(4)\oplus S(4)),\add S(3))
    \end{align*}
    in $\modu \Lambda'$ and $\modu \Lambda''$ respectively. Then by Lemma \ref{3.1}, we have a glued torsion pair

    \begin{align*}
(\mathcal{T},\mathcal{F})&=(\add({S(2)\choose S(4)}\oplus {P(1)\choose P(4)}\oplus {0\choose P(4)}\oplus {S(2)\choose P(4)}\oplus {0\choose S(4)}\oplus\\
&{0\choose P(5)}\oplus {S(1)\choose 0}\oplus {P(1)\choose S(4)}),
\add({P(1)\choose 0} \oplus {0\choose S(3)}\oplus{S(2)\choose 0}))
\end{align*}
     in $\modu \Lambda$. By Corollary \ref{cor-st}, there is a support $\tau$-tilting $\Lambda$-module $T=   {0\choose P(5)}\oplus{S(2)\choose P(4)}\oplus {0\choose P(4)}\oplus {P(1)\choose P(4)}$ such that $(\Gen T,T^{\perp_{0}})=(\mathcal{T},\mathcal{F})$.
     Obviously, $i_{*}(T')\oplus j_{!}(T'')={S(1)\choose 0}\oplus {0\choose P(5)}\oplus{S(2)\choose P(4)}\neq T$ and $i^{*}$ is not exact.

     \item[(2)] Take a support $\tau$-tilting module $T={P(1)\choose 0}\oplus {S(2)\choose 0}\oplus {S(1)\choose S(3)}\oplus {0\choose P(5)}$ in $\modu \Lambda$. It induces a torsion pair
    \begin{align*}
(\mathcal{T},\mathcal{F})=&(\add({P(1)\choose 0}\oplus {S(1)\choose 0}\oplus {S(2)\choose 0}\oplus {S(1)\choose S(3)}\oplus {0\choose S(3)}\\
&\oplus {0\choose P(5)}), \add ({0\choose P(3)}\oplus {0\choose S(4)}))
\end{align*}
 in $\modu \Lambda$.
 By Proposition \ref{con-prop-st},  $j^{*}(T)=S(3)\oplus P(5)$ is a support $\tau$-tilting module in $\modu \Lambda''$.
 Moreover, since $j_{*}j^{*}(\mathcal{F})={0\choose P(3)}\oplus {0\choose S(4)}\subseteq \mathcal{F}$,
 \begin{align*}
 (\Gen j^{*}(T),(j^{*}(T))^{\perp_{0}})&=(\add (P(5)\oplus S(3)),\add ( S(4)\oplus P(3)))\\
 &=(j^{*}(\mathcal{T}),j^{*}(\mathcal{F}))
 \end{align*}
 in $\modu \Lambda''$.

  \item[(3)] Take a support $\tau$-tilting module $T={S(2)\choose S(4)}\oplus {P(1)\choose S(4)}\oplus{0\choose S(4)}\oplus {P(1)\choose P(3)}$  in $\modu \Lambda$. It induces a torsion pair
    \begin{align*}
(\mathcal{T},\mathcal{F})=&(\add({S(2)\choose S(4)}\oplus{S(1)\choose 0}\oplus {0\choose P(3)}\oplus {P(1)\choose S(4)}\oplus {P(1)\choose P(3)}\oplus{S(1)\choose P(3)} \\
&\oplus {0\choose S(3)}\oplus {0\choose S(4)}\oplus {S(1)\choose S(3)}),
\add({0\choose P(5)} \oplus {S(2)\choose 0}\oplus{S(2)\choose P(4)}\\
& \oplus {P(1)\choose P(4)} \oplus{0\choose P(4)}\oplus {P(1)\choose 0}))
\end{align*}
 in $\modu \Lambda$.
 By Proposition \ref{con-prop-st}, $j^{*}(T)=S(4)\oplus S(4)\oplus S(4)\oplus P(3)$ is a support $\tau$-tilting module in $\modu \Lambda''$.
 Since $j_{*}j^{*}(\mathcal{F})={0\choose P(5)}\oplus {0\choose P(4)}\subseteq\mathcal{F}$,
 \begin{align*}
 (\Gen j^{*}(T),(j^{*}(T))^{\perp_{0}})&=(\add (S(4)\oplus P(3)\oplus S(3)),\add (P(5)\oplus P(4)))\\
 &=(j^{*}(\mathcal{T}),j^{*}(\mathcal{F}))
 \end{align*}
in $\modu \Lambda''$.
\item[(4)] Take $\tau$-tilting modules $T'=P(1)\oplus S(1)$ and $T''=P(5)\oplus P(3)\oplus S(3)$ in $\modu \Lambda'$ and $\modu \Lambda''$ respectively. They induce torsion pairs
    \begin{align*}
    (\mathcal{T'},\mathcal{F'})=&(\add(P(1)\oplus S(1)),\add S(2)),\\
(\mathcal{T''},\mathcal{F''})=&(\add (P(5)\oplus P(3)\oplus S(3)), \add S(4))
\end{align*}
 in $\modu \Lambda'$ and $\modu \Lambda''$ respectively.
Then by Proposition \ref{rec-tautilting},
 $T=i_{*}(T')\oplus j_{!}(T'')={P(1)\choose 0}\oplus {S(1)\choose 0}\oplus {0\choose P(5)}\oplus{P(1)\choose P(3)}\oplus {S(1)\choose S(3)}$ is a $\tau$-tilting module in $\modu \Lambda$. It induces a torsion pair
 \begin{align*}
(\mathcal{T},\mathcal{F})=&(\add ({0\choose P(5)}\oplus{P(1)\choose 0}\oplus{S(1)\choose 0}
\oplus {P(1)\choose P(3)}\oplus {S(1)\choose P(3)}\oplus {0\choose P(3)} \\
&\oplus{S(1)\choose S(3)}\oplus {0\choose S(3)}),
\add ({S(2)\choose 0}\oplus {S(2)\choose S(4)}\oplus {0\choose S(4)}))
\end{align*}
 in $\modu \Lambda$,
which is exactly the glued torsion pair with respect to $(\mathcal{T'},\mathcal{F'})$ and $(\mathcal{T''},\mathcal{F''})$.

\item[(5)] The condition ``\ $i_{*}i^{!}(\mathcal{T})\subseteq \mathcal{T}$\ '' is necessary in Proposition \ref{rec-tautilting}. Take  $\tau$-tilting modules $T'=P(1)\oplus S(1)$ and $T''=P(3)\oplus P(4)\oplus S(4)$ in $\modu \Lambda'$ and $\modu \Lambda''$ respectively.
They induce torsion pairs
    \begin{align*}
    (\mathcal{T'},\mathcal{F'})=&(\add(P(1)\oplus S(1)),\add S(2)),\\
(\mathcal{T''},\mathcal{F''})=&(\add (P(3)\oplus P(4)\oplus S(4)\oplus S(3)), \add P(5))
\end{align*}
 in $\modu \Lambda'$ and $\modu \Lambda''$ respectively. By Lemma \ref{3.1}, there is a glued torsion pair
 \begin{align*}
(\mathcal{T},\mathcal{F})=&(\add ({S(2)\choose P(4)}\oplus {P(1)\choose P(4)}\oplus{P(1)\choose 0}\oplus{0\choose P(4)}\oplus{S(2)\choose S(4)}\oplus {P(1)\choose S(4)}
\\
&\oplus {0\choose S(4)}\oplus{S(1)\choose P(3)}
\oplus {S(1)\choose S(3)}
\oplus {S(1)\choose 0}\oplus {P(1)\choose P(3)}\oplus{0\choose P(3)}\\
&\oplus {0\choose S(3)}), \ \add ({S(2)\choose 0}\oplus {0\choose P(5)}))
\end{align*}
 in $\modu \Lambda$.
 Then by Corollary \ref{cor-tau}(1), there is a $\tau$-tilting module $T={S(2)\choose P(4)}\oplus{P(1)\choose P(4)}\oplus {P(1)\choose 0}\oplus {S(2)\choose S(4)}\oplus {P(1)\choose P(3)}$ such that $(\Gen T,T^{\perp_{0}})=(\mathcal{T},\mathcal{F})$.
 Obviously, $i_{*}i^{!}(\mathcal{T})={P(1)\choose 0}\oplus {S(1)\choose 0}\oplus {S(2)\choose 0}\nsubseteq \mathcal{T}$
and
 $i_{*}(T')\oplus j_{!}(T'')={P(1)\choose 0}\oplus {S(1)\choose 0}\oplus {S(2)\choose S(4)}\oplus{P(1)\choose P(3)}\oplus {S(2)\choose P(4)}\neq T$. So $i_{*}i^{!}(\mathcal{T})\subseteq\mathcal{T}$ is necessary in Proposition \ref{rec-tautilting}.

\item[(6)] Take a $\tau$-tilting module $T={S(2)\choose S(4)}\oplus {P(1)\choose P(4)}\oplus {0\choose P(4)}\oplus {P(1)\choose S(4)}\oplus {P(1)\choose P(3)}$ in $\modu \Lambda$.
    It induces a torsion pair
    \begin{align*}
(\mathcal{T},\mathcal{F})=&(\add({S(2)\choose S(4)}\oplus{P(1)\choose P(4)}\oplus {P(1)\choose S(4)}\oplus{0\choose P(4)}\oplus {S(1)\choose 0}\oplus {P(1)\choose P(3)} \\ &\oplus {0\choose S(4)}\oplus{S(1)\choose P(3)}\oplus {S(1)\choose S(3)}\oplus {0\choose S(3)}\oplus{0\choose P(3)},
\add({0\choose P(5)}\\
&\oplus {S(2)\choose 0}\oplus {P(1)\choose 0} \oplus {S(2)\choose P(4)}))
\end{align*}
 in $\modu \Lambda$.
 By Proposition \ref{prop-t},  $j^{*}(T)=S(4)\oplus P(4)\oplus S(4)\oplus P(4)\oplus P(3)$ is a $\tau$-tilting module in $\modu \Lambda''$.
 Since $j_{*}j^{*}(\mathcal{F})={0\choose P(5)}\oplus {0\choose P(4)}\nsubseteq \mathcal{F}$,
  \begin{align*}
      (\Gen j^{*}(T),(j^{*}(T))^{\perp_{0}})&=(\add (P(4)\oplus S(4)\oplus P(3)\oplus S(3)),\add P(5)) \\
     & \neq (j^{*}(\mathcal{T}),j^{*}(\mathcal{F})).
  \end{align*}
 Obviously, $i^{*}(T)=S(1)\oplus S(1)$ is not a $\tau$-tilting module in $\modu \Lambda'$ and $i^{*}$ is not exact.
\end{itemize}}
\end{example}

\subsection*{Acknowledgements}

The first author was supported by the NSF of China (12001168), Henan University of Engineering (DKJ2019010) and the Key Research Project of Education Department of Henan Province (21A110006). The third author was supported by the NSF of China (11971225, 11901341), the project ZR2019QA015 supported by Shandong Provincial Natural Science Foundation, the project funded by China Postdoctoral Science
Foundation (2020M682141), and the Young Talents Invitation Program of Shandong Province. The authors thank the referee for the useful suggestions.


\end{document}